\documentclass[10pt,reqno]{amsart}
\usepackage[utf8]{inputenc}

\title[Rotational Symmetry II]{Rotational symmetry of solutions of mean curvature flow coming out of a double cone II}
\author{Letian Chen}
\address{Department of Mathematics, Johns Hopkins University, 3400 N. Charles Street, Baltimore, MD 21218}
\email{lchen155@jhu.edu}
\date{January 25, 2022}

\usepackage[numbers]{natbib}
\usepackage{amsmath}
\makeatletter
\renewcommand*\env@matrix[1][*\c@MaxMatrixCols c]{%
\hskip -\arraycolsep
\let\@ifnextchar\new@ifnextchar
\array{#1}}
\makeatother
\usepackage{amsfonts}
\usepackage{amssymb}
\usepackage{mathrsfs}
\usepackage{mathtools}
\usepackage{enumitem}
\usepackage{url}
\usepackage[pagebackref=true]{hyperref}

\usepackage[noabbrev,capitalize,nameinlink]{cleveref}
\usepackage{flafter}

\usepackage{amsthm}
\newtheorem{thm}{Theorem}[section]
\newtheorem{lem}[thm]{Lemma}
\newtheorem{prop}[thm]{Proposition}
\newtheorem{cor}[thm]{Corollary}
\newtheorem{conj}[thm]{Conjecture}
\newtheorem{ques}[thm]{Question}

\theoremstyle{definition}
\newtheorem{defn}[thm]{Definition}

\theoremstyle{remark}
\newtheorem*{rem}{Remark}

\newcommand{\overbar}[1]{\mkern 1.5mu\overline{\mkern-1.5mu#1\mkern-1.5mu}\mkern 1.5mu}
\newcommand\abs[1]{\left|#1\right|}
\newcommand\norm[1]{\left\lVert#1\right\rVert}
\newcommand\inner[1]{\left \langle #1\right \rangle}

\DeclareMathOperator{\Div}{div}

\DeclareMathOperator{\supp}{supp}

\DeclareMathOperator{\dist}{dist}
\DeclareMathOperator{\reg}{reg}

\numberwithin{equation}{section}

\crefformat{equation}{(#2#1#3)}
\crefformat{enumi}{(#2#1#3)}

\begin{document}
\begin{abstract}
	Given a double cone $\mathcal{C}$ with entropy at most two that is symmetric across some hyperplane, we show that any integral Brakke flow coming out of the cone must inherit the reflection symmetry for all time, provided the flow is smooth for a short time. As a corollary we prove that any such flow coming out of a rotationally symmetric double cone must stay rotationally symmetric for all time. We also show the existence of a non-self-similar flow coming out of a double cone with entropy at most two, and give an example of such a flow with a finite time singularity. Additionally, we show the existence of self-expanders with triple junctions, which are exceptions to our main theorem.
\end{abstract}
\maketitle
\section{Introduction}
A family of properly embedded hypersurface $\{\Sigma_t\}_{t \in I}$ satisfies the mean curvature flow (MCF) equation if: 
\begin{align*}
	\left(\frac{\partial x}{\partial t}\right)^\perp = H_{\Sigma_t}(x), \;\; x \in \Sigma_t.
\end{align*}
Here $H_{\Sigma_t}$ is the mean curvature vector of $\Sigma_t$, $x$ is the position vector and $\perp$ denotes the projection onto the normal bundle. In this paper we study mean curvature flows (MCF) that come out of (smooth) cones. Due to the singularity at the origin, there could be multiple distinct mean curvature flows coming out of a given cone $\mathcal{C}$. We are interested in how symmetries of the cone influence these solutions.  \par 
The simplest solutions coming out of cones are self-expanders. We say a properly embedded hypersurface $\Sigma^n \subset \mathbb{R}^{n+1}$ is a \textit{self-expander} if 
\begin{align}
\label{self-expander-equation}
	H_\Sigma(x) = \frac{x^\perp}{2}.
\end{align}
Equivalently, $\Sigma$ is an self-expander if and only if $\{\sqrt{t}\Sigma\}_{t \in (0,\infty)}$ is a solution to the MCF. Self-expanders are of particular importance in the study of singularities of MCF as they arise naturally as models of how MCFs flow through conical singularities. \par 
There exist, however, other solutions coming out of the cone that are not self-similar. A class of such solutions was first sketched by Bernstein--Wang \cite{BWTopologicalUniqueness} as Morse flow lines between unstable and stable self-expanders asymptotic to the same cone. Note that the existence of such solution requires the existence of an unstable self-expander asymptotic to $\mathcal{C}$, which is not guaranteed.  \par 
In the present article, we study flows coming out of a double cone with reflection symmetry. Our main result roughly says that any MCF coming out of a suitable cone with a reflection symmetry across some hyperplane must inherit the symmetry for all future times, provided the flow is initially smooth for a short time (see \cref{preliminaries} for terminologies). 
\begin{thm}
	\label{reflection-symmetry}
	Let $\Pi \subset \mathbb{R}^{n+1}$ be a hyperplane passing through the origin and $\mathbb{H}$ an open half-space with $\partial \mathbb{H} = \Pi$. Let $\mathcal{C} \subset \mathbb{R}^{n+1}$ be a smooth double cone with $\lambda[\mathcal{C}] < 2$ such that $\mathcal{C} \cap \mathbb{H}$ is a Lipschitz graph over $\Pi$. Let $\mathcal{M} = \{\mu_t\}_{t \in [0,\infty)}$ be an integral, unit-regular and cyclic Brakke flow coming out of $\mathcal{C}$; that is,
	\begin{align*}
		\lim_{t \to 0} \mu_t = \mathcal{H}^n \llcorner \mathcal{C}
	\end{align*} 
	as Radon measures. Suppose also $\mathcal{M}$ is smooth on $(0,T)$ for some $T > 0$. If $\mathcal{C}$ is symmetric across $\Pi$, then so is $\mathcal{M}$ for $t \in [0,\infty)$. Moreover, $\mathcal{M}$ is smooth away from $\Pi$.
\end{thm}
Applying the above to rotationally symmetric double cones we obtain:
\begin{cor}
	\label{main-theorem}
	Let $\mathcal{C} \subset \mathbb{R}^{n+1}$ be a smooth rotationally symmetric double cone with $\lambda[\mathcal{C}] < 2$ (see \cref{preliminaries} for the precise definitions). Let $\mathcal{M} = \{\mu_t\}_{t \in [0,\infty)}$ be an integral, unit-regular and cyclic Brakke flow coming out of $\mathcal{C}$ that is smooth on $(0,T)$. Then $\mathcal{M}$ is rotationally symmetric with the same axis of symmetry as $\mathcal{C}$. Moreover, $\mathcal{M}$ is smooth away from its axis of symmetry. 
\end{cor} 
\begin{rem}
	In our previous work \cite{Chen}, we only showed that the flow is rotationally symmetric up to the first singular time $T$, so the point of the current work is to show that the symmetry still holds after any singularity, which, as part of the conclusion, must lie on the axis of symmetry.
\end{rem}
\begin{rem}
	The entropy condition $\lambda[\mathcal{C}] < 2$ is likely redundant. See \cref{further-remarks}.
\end{rem}
In fact, if the cone $\mathcal{C} \subset \mathbb{R}^{n+1}$ is of the form
\begin{align}
	\label{good-cone}
	x_1^2 = m^2\left(x_2^2 + \cdots + x_{n+1}^2\right), \;\; m > 0,
\end{align}
where $m$ is the parameter determined by the cone angle, a much stronger conclusion holds:
\begin{cor}
	\label{cylindrical-singularity}
	Suppose $\mathcal{C}$ is of the form \cref{good-cone} and has entropy $\lambda[\mathcal{C}] < 2$. Suppose $\mathcal{M}$ is an integral, unit-regular and cyclic Brakke flow coming out of $\mathcal{C}$. If $\mathcal{M}$ is smooth on $(0,T)$, then $\mathcal{M}$ is rotationally symmetric across the $x_1$-axis. The only possible singularity model of $\mathcal{M}$ is the round cylinder $\mathbb{R} \times \mathbb{S}^{n-1}$. Moreover, there can be at most one of such singularity, which, if it exists, must occur at the origin. 
\end{cor}
We can also apply \cref{reflection-symmetry} to cones with $O(p + 1) \times O(n-p+1)$ symmetry to prove:
\begin{cor}
	\label{simons-cone-thm}
	Suppose $n \ge 2$ and $1 \le p \le n-1$. Let $\mathcal{C} \subset \mathbb{R}^{n+2}$ be a cone invariant under $O(p+1) \times O(n-p+1)$ with $\lambda[\mathcal{C}] < 2$. Let $\mathcal{M} = \{\mu_t\}_{t \in [0,\infty)}$ be an integral, unit-regular and cyclic Brakke flow coming out of $\mathcal{C}$. If there is $T > 0$ such that $\mathcal{M}$ is smooth on $(0,T)$, then $\mathcal{M}$ inherits the $O(p+1) \times O(n-p+1)$ symmetry (with the same axes of symmetry).
\end{cor}
There are many other related works on the rotational symmetry of self-expanders. Observe that if there is a unique self-expander asymptotic to a rotationally symmetric cone, then it must inherit the rotational symmetry. The first nontrivial result for double cones is obtained by Fong--McGrath \cite{FongMcGrath}. They proved that a mean-convex self-expander (i.e. $H_\Sigma > 0$) asymptotic to a rotationally symmetric double cone is rotationally symmetric. This is later generalized by Bernstein--Wang (Lemma 8.3 in \cite{BWIntegerDegree}) to weakly stable self-expanders. In our previous work \cite{Chen}, we showed in full generality that any smooth self-expander asymptotic to a rotationally symmetric double cone is rotationally symmetric. Rotational symmetry of self-expanding solitons of other geometric flows has been studied by Chodosh \cite{Chodosh} and Chodosh--Fong \cite{ChodoshFong}. \par
Next we briefly comment on some of the assumptions made in \cref{reflection-symmetry}. First of all, the only extra assumption over the smooth case is the entropy bound $\lambda[\mathcal{C}] < 2$. This is due to the complicated nature of singularities of higher multiplicities arising from Brakke flows. In particular, the entropy condition is essential to the maximum principle \cref{maximum-principles} (in which the entropy controls the Gaussian density) and the proof of \cref{tameness}. \par 
Secondly, the cyclicity of $\mathcal{M}$ is needed to ensure singularities modeled on triple junctions do not appear along the flow. Standard Schauder estimates and pseudolocality arguments can only guarantee smoothness outside of a large ball (see for example \cref{interior-estimates}) but do not rule out formations of triple junctions. This condition is explicitly used in \cref{tameness}. In fact, we will show that there exists a self-expander with triple junctions in \cref{ode-appendix}. Our moving plane method will not work for such self-expanders. However, if we assume that $\mathcal{C}$ is also symmetric across the hyperplane perpendicular to its axis of symmetry, then there is still some hope that all self-expanders with triple junctions are rotationally symmetric (See \cref{further-remarks}). \par 
Finally, it is not immediately clear that our theorem contains more flows than \cite{Chen}, which includes \textit{smooth} self-expanders and low entropy flow lines of Bernstein--Wang \cite{BWTopologicalUniqueness}. For this reason we will establish an existence result of a non-self-similar flow coming out of a cone $\mathcal{C}$ with $\lambda[\mathcal{C}] < 2$. We warn the readers that, while the flow constructed below agrees with the construction from \cite{BWTopologicalUniqueness} on the smooth part, they are not necessarily the same flow. This is due to the lack of understanding how expander mean convexity is preserved after the singularity.
\begin{thm}
	\label{nontrivial-flow-lines}
	Let $\mathcal{C} \subset \mathbb{R}^{n+1}$ be a smooth cone with $\lambda[\mathcal{C}] < 2$. Let $\Sigma$ be a self-expander $C^{2,\alpha}$-asymptotic to $\mathcal{C}$. If $\Sigma$ is unstable, there exists a non-self-similar immortal integral Brakke flow $\mathcal{M} = \{\mu_t\}_{t \in (0,\infty)}$ such that 
	\begin{align*}
		\lim_{t \to 0} \mu_t = \mathcal{H}^n \llcorner \mathcal{C}.
	\end{align*}
	Moreover, there is $T > 0$ such that $\mathcal{M}$ is smooth on $(0,T)$.
\end{thm}
\begin{rem}
	This existence theorem is weaker than the one obtained in \cite{BCW} in which no entropy assumption is made, but is enough for our purposes. Thanks to the entropy bound, it is much simpler to analyze our situation as the flows we produce are automatically matching. A similar construction for self-shrinkers is carried out by Chodosh--Choi--Mantoulidis--Schulze \cite{CCMSGeneric}. 
\end{rem}
\begin{rem}
	Without the more restrictive entropy bound, it is not enough to produce the flow line as a smooth MCF --- Some unstable connected self-expanders under the flow will necessarily disconnect in order to reach the stable double-disk solution. See \cref{cylindrical-singularity} and \cref{disconnection}. 
\end{rem}

Let us now briefly discuss the proof of \cref{reflection-symmetry}. As in \cite{Chen}, the proof is based on the moving plane method, first used by Alexandrov to prove that embedded, compact CMC hypersurfaces are round spheres. The method was then employed by Gidas--Ni--Nirenberg \cite{GNNSymmetry} who proved radial symmetry of positive solutions to the Yamabe problem, and by Schoen \cite{Schoen} who proved the uniqueness of catenoids. Recently, there are a number of other results on MCF that utilize the moving plane method, including \cite{MSS} and \cite{CHHAncient}. Unfortunately since we are dealing with potentially singular flows, these methods for smooth flows are no longer sufficient. Instead we will use a novel variant without smoothness, recently developed by \cite{CHHW} (for parabolic equations) and \cite{HHW} (for elliptic equations). The key ingredient in the proof is the Hopf lemma without smoothness \cref{hopf-lemma}, which allows us to upgrade regularity of the flow concurrently with the symmetry. \par 
In order to apply the (geometric) moving plane method to noncompact objects, it is mandatory to have a well-controlled asymptote at infinity. In the asymptotically cylindrical case \cite{CHHW} (see also the preceding \cite{CHHAncient}), a fine neck analysis is carried out in order to determine the asymptotic expansion near infinity. However, to our advantage, in the asymptotically conical case, the asymptote near infinity is entirely determined by the given cone $\mathcal{C}$, and it is an immediate consequence of the pseudolocality theorem that our flows are nice $C^{2,\alpha}$ normal graphs over $\mathcal{C}$ outside of a large ball with appropriate decay rates (see \cref{smooth-construction-appendix} or the works of Bernstein--Wang \cite{BWSmoothCompactness} or \cite{BWTopology}). Knowing this, we can then carry out the moving plane method with the usual maximum principle and Hopf lemma replaced by \cref{maximum-principles} and \cref{hopf-lemma} respectively (barring some technicality, e.g. the tameness assumption \cref{tameness}). \par 
We will prove \cref{reflection-symmetry} in \cref{reflection-symmetry-section}. In \cref{rotational-symmetry-section} we apply \cref{reflection-symmetry} to prove the various corollaries mentioned above. In \cref{construction-of-matching-motion} we give a construction for  \cref{nontrivial-flow-lines}. In \cref{ode-appendix} we will prove an ODE existence result of self-expanders with triple junctions, which shows that the cyclic assumption in \cref{reflection-symmetry} is indeed necessary. In \cref{smooth-construction-appendix} we review the construction of smooth Morse flow line following \cite{BWTopologicalUniqueness}. Finally in \cref{maximum-principles-appendix} we recall the key maximum principle and Hopf lemma from Section 3 of \cite{CHHW}. \par 
\subsection*{Acknowledgment}
I would like to thank my advisor, Jacob Bernstein, for many helpful comments on the paper and continuous encouragement. I would also like to thank Junfu Yao for helpful conversations, Kyeongsu Choi and Or Hershkovits for explaining some arguments in \cite{CHHW}, and the referee for many constructive comments on the first draft of this article.

\section{Preliminaries}
\label{preliminaries}
\subsection{Notations} Throughout the paper lower case letters such as $x$ denote points in $\mathbb{R}^{n+1}$, while upper case letters such as $X$ denote points in the spacetime $\mathbb{R}^{n+1} \times [0,\infty)$. $B_r(x)$ denotes the Euclidean ball of radius $r$ centered at $x$, and $P(X,r)$ denotes the parabolic ball centered at $X = (x,t)$ of radius $r$, i.e.
\begin{align*}
	P(X,r) = B_r(x) \times (t - r^2,t].
\end{align*}
$\mathcal{T}_r(A)$ denotes the tubular neighborhood of $A \subset \mathbb{R}^{n+1}$ of radius $r$. Finally, for $x = (x', x_{n+1})$, let 
\begin{align*}
	B_r^n(x) = \{(y',y_{n+1}) \in \mathbb{R}^{n+1} \mid \abs{y' - x'} < r, y_{n+1} = x_{n+1}\},
\end{align*}
and $C_r(x)$ be the open cylinder of height $r$ over $B_r^n(x)$, i.e.
\begin{align*}
	C_r(x) = \{(y',y_{n+1}) \in \mathbb{R}^{n+1} \mid \abs{x' - y'} < r, \abs{x_{n+1} - y_{n+1}} < r\}.
\end{align*}
\par 
By a \textit{(hyper)cone} we mean a set $\mathcal{C} \subset \mathbb{R}^{n+1}$ that is invariant under dilation, i.e. $\rho \mathcal{C} = \mathcal{C}$ for all $\rho > 0$. The \textit{link} of $\mathcal{C}$ is $\mathcal{L}(\mathcal{C}) = \mathcal{C} \cap \mathbb{S}^n$. We say $\mathcal{C}$ is smooth if $\mathcal{L}(\mathcal{C})$ is a smooth hypersurface of $\mathbb{S}^n$. A double cone is a cone $\mathcal{C}$ whose link $\mathcal{L}(\mathcal{C})$ has two connected components lying in opposite hemispheres of $\mathbb{S}^n$. A hypersurface $\Sigma$ is \textit{$C^{k,\alpha}$-asymptotically conical} to $\mathcal{C}$ if 
\begin{align*}
	\lim_{\rho \to 0^+} \rho \Sigma = \mathcal{C} \text{ in } C^{k,\alpha}_{loc}(\mathbb{R}^{n+1} \setminus \{0\}).
\end{align*}
We will simply say $\Sigma$ is asymptotically conical to $\mathcal{C}$ if it is smoothly asymptotically conical. In our applications, the cone $\mathcal{C}$ is almost always assumed to be smooth and the asymptotically conical hypersurfaces will come from solutions to MCF, which are in fact smoothly asymptotic to $\mathcal{C}$ by standard Schauder theory. \par 
Given a hypersurface $\Sigma \subset \mathbb{R}^{n+1}$, the \textit{Gaussian surface area} of $\Sigma$ is 
\begin{align*}
	F[\Sigma] = \int_\Sigma e^{-\frac{\abs{x}^2}{4}} d\mathcal{H}^n.
\end{align*}
Following Colding--Minicozzi \cite{CMGeneric}, the \textit{entropy} of $\Sigma$ is 
\begin{align*}
	\lambda[\Sigma] = \sup_{\rho \in \mathbb{R}^+, x_0 \in \mathbb{R}^{n+1}} F[\rho \Sigma + x_0] = \sup_{\rho, x_0} \frac{1}{(4\pi \rho)^{n/2}} \int_{\Sigma}  e^{-\frac{\abs{x-x_0}^2}{4\rho}} d\mathcal{H}^n.
\end{align*}
By Huisken's monotonicity formula, the entropy is nonincreasing along a MCF.

\subsection{Integral Brakke flows} 
\label{brakkeflow}
For the rest of the article we will be using notations from geometric measure theory. We refer to \cite{SimonBook} and \cite{Ilmanen} for the relevant definitions. \par
Since we are dealing with potentially singular MCF, we need to generalize the classes of MCF to make sense of the flow and hence the symmetries past singularities. The measure-theoretic generalization of MCF is the Brakke flow \cite{Brakke} (see also the more recent book of Tonegawa \cite{Tonegawa}), which is a flow of varifolds. Given an integral $n$-rectifiable Radon measure $\mu$, let $V(\mu)$ denote the associated integral varifold, and $H$ its generalized mean curvature vector given by the formula:
\begin{align*}
	\int \Div_{V(\mu)} X d\mu = - \int H \cdot X d\mu.
\end{align*}
where $X$ is a compactly supported $C^1$ vector field. \par 
Given an open set $U \subset \mathbb{R}^{n+1}$, by an \textit{integral $n$-Brakke flow} in $U$ we mean a family of integral $n$-rectifiable Radon measures $\mathcal{M} = \{\mu_t\}_{t \in I}$ such that:
\begin{enumerate}[label = (\alph*)]
	\item For a.e. $t \in I$, $V(\mu)$ has locally bounded first variation and its generalized mean curvature vector $H$ is orthogonal to the approximating tangent space of $V(\mu)$ $x$-a.e.
	\item For any bounded interval $[a,b] \subset I$ and compact set $K \subset U$,
	\begin{align*}
		\int_a^b \int_K (1 + \abs{H}^2) d\mu_tdt < \infty.
	\end{align*}
	\item For $[a,b] \subset I$ and every $\phi \in C_c^1(U \times [a,b], \mathbb{R}^+)$,
	\begin{align*}
		\int \phi d\mu_b - \int \phi d\mu_a \le \int_a^b \int \left(-\phi\abs{H}^2 + \abs{H} \cdot \nabla \phi + \frac{\partial \phi}{\partial t}\right)d\mu_tdt.
	\end{align*}
\end{enumerate}
Since we are working in codimension one we will drop the dependence on $n$ in the definition above and simply refer to it as a Brakke flow. \par 
Brakke flow has two main drawbacks. First, due to the inequality in condition (c), a Brakke flow can vanish abruptly (in fact some evolution must involve such vanishing). In order to avoid this technical difficulty in \cref{construction-of-matching-motion} we will employ the notion of a matching flow which, in some sense, prevents certain sudden loss of mass. Secondly, and as part of our motivation, Brakke flow does not have to be unique. In fact, there could be multiple (smooth) self-expanders coming out of a fixed cone, which is singular at the origin. \par 

A Brakke flow $\mathcal{M}$ is \textit{unit-regular} if, for a spacetime point $X = (x,t)$, $\mathcal{M}$ is smooth and has no sudden mass loss if a tangent flow at $X$ is a multiplicity one hyperplane. We say $\mathcal{M}$ is \textit{cyclic} if the associated mod-2 flat chain $[V(\mu_t)]$ (see eg. \cite{WhiteCurrentsChains}) has no boundary. By works of Ilmanen \cite{Ilmanen}, Brakke flows  produced by elliptic regularization are unit-regular and cyclic. \par 
Throughout our presentation, given a Brakke flow $\mathcal{M} = \{\mu_t\}_{t \in I}$, we will write $\mathcal{M}_t = \supp \mu_t$ for $t \in I$. 

\section{Reflection symmetry}
\label{reflection-symmetry-section}
In this section we prove \cref{reflection-symmetry}. Fix a smooth double cone $\mathcal{C}$ with $\lambda[\mathcal{C}] < 2$ and an integral, unit-regular Brakke flow $\mathcal{M} = \{\mu_t\}$ satisfying the assumptions of \cref{reflection-symmetry}; that is, 
\begin{align*}
	\lim_{t \to 0} \mu_t = \mathcal{H}^n \llcorner \mathcal{C},
\end{align*}
and $\mathcal{M}$ is smooth on $(0,T)$ for some $T > 0$. Given a hyperplane $\Pi$, if $\mathcal{C} \cap \mathbb{H}$ is graphical over $\Pi$ and symmetric across $\Pi$, the main theorem in our previous work \cite{Chen} implies that $\mathcal{M}$ is  symmetric across $\Pi$ until its first singular time, past which the usual moving plane argument stops to work. For this reason, we will prove \cref{reflection-symmetry} using a version of the moving plane method without assuming smoothness, recently developed by \cite{CHHW} (see also \cite{HHW}). \par 
A technical ingredient we need for the moving plane method is the notion of tameness, which we now define.

\iffalse 
At the first singular time, any singularity $X$ must be on the axis of symmetry, and so any tangent flow at $X$ must also be rotationally symmetric. Since the flow is not closed and $\lambda[\mathcal{C}] < 2$, it is not hard to see that such a tangent flow must be a multiplicity one cylinder. Knowing this, we suspect it is possible to extend the argument in the smooth case up to the second singular time. Beyond the second singular time, however, we can not rule out other types of singularity anymore. For example, it is possible that the flow $\mathcal{M}$ forms two necks simultaneously at $T$ and subsequently disconnect into three connected components, one of which will be closed and disappears at a compact singularity. \par 
\fi

\begin{defn}[Definition 3.1 in \cite{CHHW}]
	\label{tameness-defn}
	For an integral Brakke flow $\mathcal{M}$ in $\mathbb{R}^{n+1}$, we say $X \in \mathcal{M}$ is a tame point of the flow if the $-1$ time slice of every tangent flow at $X$ is smooth with multiplicity one away from a singular set $\mathcal{S}$ with $\mathcal{H}^{n-1}(\mathcal{S}) = 0$. We say $\mathcal{M}$ is a tame flow if every point $X \in \mathcal{M}$ is a tame point.
\end{defn}
For instance, a tame flow should not have a singularity modeled on a triple junction (that is, three hyperplanes meeting at equal angles) or a multiplicity two plane. Tameness is a key assumption to apply the Hopf lemma without smoothness \cref{hopf-lemma}, which, in turn, is crucial to the moving plane method. The next proposition establishes tameness of $\mathcal{M}$.
\begin{prop}
	\label{tameness}
	Let $\mathcal{M}$ be as above. Then $\mathcal{M}$ is a tame flow.
\end{prop}
\begin{proof}
	It suffices to check the definition. Let $\mathcal{X} = \{\nu_t\}_{t \in (-\infty,0]}$ be a tangent flow at $(x_0,t_0)$. Since $\lambda[\mathcal{C}] < 2$, $\mathcal{X}$ has multiplicity 1 (i.e. the Gaussian density is 1 $\mathcal{H}^{n-1}$-a.e. on $\mathcal{X}$, $t$-a.e.). \par 
	If $\mathcal{X}$ is static or quasi-static, then $\nu_{-1} = \mathcal{H}^n \llcorner \Gamma$ for some stationary cone $\Gamma$. If $\Gamma$ splits off $(n-2)$-lines, then $\nu_{-1} = \mu_{\mathbb{R}^{n-2}} \times \nu'$ where $\nu'$ is a one-dimensional stationary cone in $\mathbb{R}^2$. Hence $\nu'$ is a union of half-rays. Since $\lambda[\nu'] = \lambda[\nu_{-1}] < 2$, there are at most 3 rays. Since $\mathcal{M}$ is cyclic, $\nu'$ cannot be 3 rays meeting at the origin. This is because the triple junction is not cyclic, and a cyclic Brakke flow cannot have a singularity modeled on a non-cyclic singularity by \cite{WhiteCurrentsChains}. Therefore $\nu'$ consists of 2 lines and in fact $\nu_{-1}$ is smooth. So any singular cone $\nu_{-1}$ can split off at most $(n-3)$-lines, and consequently the singular part of $\nu_{-1}$ has codimension at least 3. \par 
	If $\mathcal{X}$ is a non-flat self-shrinker, then any tangent cone $\nu'$ to $\nu$ is a stationary cone with entropy at most 2 (here we used the fact that self-shrinkers are minimal surfaces with respect to the metric $g_{ij} = e^{-\frac{\abs{x}^2}{2n}}\delta_{ij}$ which is conformal to the Euclidean metric). It follows from the above discussion that $\nu'$ can split off at most $(n-3)$-lines (if $\nu'$ splits of $(n-2)$-lines then it is a multiplicity 1 hyperplane, which, by Allard regularity theorem, means that $\mathcal{S}(\mathcal{X}) \subset \{0\}$ and consequently $\mathcal{X}$ is a multiplicity 1 hyperplane), and so $\nu_{-1}$ is smooth away from a set of Hausdorff dimension at most $(n-2)$. Thus $\mathcal{M}$ is tame.
\end{proof}
\begin{rem} With a more restrictive entropy bound it is possible to refine the codimension of the singular set even more. See \cite{BWSharpLower} or Section 4 of \cite{CHHW}.
\end{rem}
Next we establish some properties of $\mathcal{M}$. The next proposition says that the flow $\mathcal{M}$ stays asymptotically conical for all future time.
\begin{prop}
	\label{asymptotic-conical}
	Let $\mathcal{M}$ be as above. Then $\mathcal{M}$ is asymptotically conical to $\mathcal{C}$ for all $t \in (0,\infty)$. Consequently, for every $t > 0$ there is $R = R(t)$ such that $\mathcal{M}_t \setminus B_R(0)$ is a smooth MCF.
\end{prop}
\begin{proof}
	This follows from pseudolocality for Brakke flows and parabolic Schauder estimates as in \cref{interior-estimates} (see also Proposition 4.4 of \cite{BWTopology}). 
\end{proof}
Fix an open half space $\mathbb{H}$ and let $\Pi = \partial \mathbb{H}$. We now use the pseudolocality theorem to prove that if $\mathcal{C}$ is graphical over $\Pi$, then so is $\mathcal{M}$ outside of a large compact set. This will serve as an asymptotic expansion at infinity and be upgraded into interior graphicality via the moving plane method. For the next two lemmas we write $\mathcal{M}_t^+ = \mathcal{M}_t \cap \mathbb{H}$.
\begin{lem}
\label{asymptotic-expansion}
	Suppose $\mathcal{C} \cap \mathbb{H}$ is a Lipschitz graph over $\Pi$.  let $\mathcal{M}$ be as above, then for every $t > 0$ there is $R = R(t)$ and a smooth function $u$ on $\mathcal{C}$ such that 
	\begin{align*}
		\mathcal{M}_t^+ \setminus B_R(0) \subset \{p + u(p)\nu_\mathcal{C}(p) \mid p \in \mathcal{C}\}. 
	\end{align*}
	and $\abs{u(p)} \le C \abs{p}^{-1}$ for $p \in \mathcal{C}$ for some constant $C = C(t)$. 
\end{lem}
\begin{proof}
	Fix a time $t_0$. Let us first show that $\mathcal{M}_t^+ \setminus B_R(0)$ can be written as a smooth normal graph over $\mathcal{C}$.  By \cref{asymptotic-conical}, there exists $R = R(t) > 0$ such that $\mathcal{M}_t \setminus B_R(0)$ is asymptotically conical to $\mathcal{C}$. Since $\mathcal{M}_0 = \mathcal{C}$, by pseudolocality theorem for Brakke flows (Theorem 1.5 of \cite{IlmanenNevesSchulze}), given $\eta > 0$ there exists $t_1$ such that for $0 < t < t_1$ and $x \in \mathcal{C} \setminus B_1(0)$, $\mathcal{M}_t \cap C_{\sqrt{t_1}}(x)$ can be written as a normal graph over $B_{\sqrt{t_1}}^n(x) \cap T_{x}\mathcal{C}$ with Lipschitz constant bounded by $\eta$. By parabolic rescaling, we see that, for $0 < t < 2t_0$ and $x \in \mathcal{C} \setminus B_{\sqrt{2t_0t_1^{-1}}}(0)$, $\mathcal{M}_t \cap C_{\sqrt{2t_0}}(x)$ can be written as a normal graph over $B_{\sqrt{2t_0}}^n(x)$ with Lipschitz constant bounded by $\eta$. In particular putting $t = t_0$ gives the desired graphicality. The regularity of $u$ follows from \cref{asymptotic-conical}. \par 
	To see that the function $u$ decays near infinity, by a similar argument as in \cref{distance-estimate}, there exists $N$ such that for all $R > 1$ we have
	\begin{align*}
		\mathcal{M}_{t_0} \setminus B_{NR\sqrt{t_0+1}}(0) \subset \mathcal{T}_{R^{-1}\sqrt{t_0+1}}(\mathcal{C}).
	\end{align*}	
	Equivalently, for $R > N\sqrt{t_0+1}$,
	\begin{align*}
		\mathcal{M}_{t_0} \setminus B_{R'}(0) \subset \mathcal{T}_{N(t_0+1)(R')^{-1}}(\mathcal{C}).
	\end{align*}
	Enlarge $R$ if needed so that $\mathcal{M}_{t_0}^+ \setminus B_R(0)$ is a normal graph over $\mathcal{C}$. We see that $u$ satisfies $\abs{u(p)} \le C\abs{p}^{-1}$ for $p \in \mathcal{C}$ and $u(p) \in \mathcal{M}_t^+ \setminus B_{2R}(0)$.
\end{proof}
\begin{lem}
\label{graphicality}
	Suppose $\mathcal{C} \cap \mathbb{H}$ is a Lipschitz graph over $\Pi$. Let $\mathcal{M}$ be as above. For every $t > 0$ there is $R = R(t)$ such that $\mathcal{M}_t^+ \setminus B_R(0)$ can be written as a graph over $\Pi$; that is, the projection $\pi: \mathcal{M}_t^+ \setminus B_R(0) \to \Pi$ is injective.
\end{lem}
\begin{proof}
	By \cref{asymptotic-expansion}, for every $\eta > 0$, there is $R = R(t)$ such that $\mathcal{M}_t^+ \setminus B_R(0)$ is a normal graph over $\mathcal{C}$ with Lipschitz constant bounded by $\eta$. Since $\mathcal{C} \cap \mathbb{H}$ is a Lipschitz graph over $\Pi$, the unit normal vector $\nu_\Pi$ is not contained in any tangent space to $x' \in (\mathcal{C} \cap \mathbb{H}) \setminus \{0\}$. Therefore by taking $\eta$ sufficiently small we may make sure that $\nu_\Pi$ is also not contained in any tangent space to $x \in \mathcal{M}_t^+ \setminus B_R(0)$ (here $R = R(t,\eta)$, but of course $\eta$ in turn depends on $t$). This proves that $\mathcal{M}_t^+ \setminus B_R(0)$ is graphical over $\Pi$ as well.
\end{proof}

\cref{tameness}  and \cref{graphicality} allow us to use the moving plane method without smoothness, which we now carry out. Let 
\begin{align*}
	\Pi^s = \{(x,x_{n+1}) \in \mathbb{R}^{n+1} \mid x_{n+1}= s\} \times [0,\infty) \subset \mathbb{R}^{n+1} \times [0,\infty)
\end{align*} 
be the hyperplane at level $s$ in spacetime.  Given a set $A \subset \mathbb{R}^{n+1} \times [0,\infty)$ and $s \in [0,\infty)$ we let 
\begin{align*}
	A^{s+} = \{(x,x_{n+1},t) \in A \mid x_{n+1} >  s\} \text{ and } A^{s-} = \{(x,x_{n+1},t) \in A \mid x_{n+1} < s\}
\end{align*}
be the parts of $A$ lying above $\Pi_s$ and below $\Pi_s$ respectively. Finally, the set 
\begin{align*}
	A^{s*} = \{(x,x_{n+1},t) \mid (x, 2s - x_{n+1},t) \in A\}
\end{align*}
is the reflection of $A$ across $\Pi_s$. We say $A > B$ for $A, B \subset \mathbb{R}^{n+1} \times [0,\infty)$ provided for any $(x,s,t) \in A$ and $(x,s',t) \in B$ we have $s > s'$. In contrast, a subscript $t$ will continue to denote the time $t$ slice of a spacetime set. \par 
To set up the proof of \cref{reflection-symmetry}, WLOG we may assume the hyperplane is $\{x_{n+1} = 0\}$. Fix a time $T_0 > 0$. We consider $\mathcal{M}$ on $[0,T_0)$ as its spacetime track; namely,
\begin{align*}
	\mathcal{M} = \bigcup_{t =0}^{T_0} \mathcal{M}_t \times \{t\} \subset \mathbb{R}^{n+1} \times [0,T_0]. 
\end{align*}
Finally, let
\begin{align*}
	S = \{s \in (0,\infty) \mid (\mathcal{M}^{s+})^* > \mathcal{M}^{s-}, \text{ and } (\mathcal{M}^{s+})_t \text{ is graphical over } (\Pi^s)_t \text{ for } t \in [0,T_0]\}.
\end{align*}
Here graphicality means that the projection $\pi_s: (\mathcal{M}^{s+})_t \to (\Pi^s)_t$ is injective for $t \in [0,T]$. Since each $(\mathcal{M}^{s+})_t$ is countably $n$-rectifiable, graphicality is equivalent to that the unit normal $e_{n+1} = (0,\ldots,0,1)$ of $(\Pi^s)_t$ is not contained in the approximate tangent space of $(\mathcal{M}^{s+})_t$ for $t \in [0,T_0]$. Observe that $(\mathcal{M}^{s+})^*$ is asymptotically conical to the translated cone $(\mathcal{C} + 2se_{n+1}) \times [0,\infty)$ (in the sense of \cref{asymptotic-expansion} --- this ensures that a hypersurface cannot be simultaneously asymptotic to two distinct cones). \par
First we need a lemma about smoothness of the top part of the flow similar to Proposition 7.4 of \cite{CHHW}.
\begin{lem}
	\label{smoothness}
	Suppose $s > 0$ and $s \in S$. Then $\mathcal{M}^{s+}$ is a smooth MCF asymptotic to $\mathcal{C}$. Moreover, every point on $\mathcal{M} \cap \{x_{n+1} = s\}$ is a regular point of the flow. 
\end{lem}
\begin{proof}
	Let 
	\begin{align*}
		I_s = \{s' \ge s \mid \mathcal{M}^{s'+} \text{ is smooth} \}.
	\end{align*}
	By \cref{asymptotic-conical}, there is $R = R(t)$ such that $\mathcal{M}_t \setminus B_{R}(0)$ is a smooth MCF asymptotic to $\mathcal{C}$. So for sufficiently large $s$ depending on $T_0$ we see that $\mathcal{M}^{s+}$ is asymptotic to $\mathcal{C} \times [0,T_0]$. This shows $I_s$ is not empty. \par 
	Let $s_0 = \inf I_s$. We first argue that $\mathcal{M} \cap \{x_{n+1} = s_0\}$ consists of regular points. Let $(x_0,t_0) \in \mathcal{M} \cap \{x_{n+1} = s_0\}$ and let $\mathcal{M}^*$ be the flow reflected across $\Pi_s$. We wish to apply the Hopf lemma \cref{hopf-lemma} to $\mathcal{M}$, $\mathcal{M}^*$ and $\mathbb{H} = \{x_n < s\}$ to conclude that $(x_0,t_0)$ is a regular point. To this end we must check that the conditions are satisfied. Tameness follows from \cref{tameness}. We may also assume that $\partial \mathbb{H}$ is not a tangent flow to either $\mathcal{M}$ or $\mathcal{M}^*$ at $(x_0,t_0)$, because otherwise the entropy bound together with Brakke regularity theorem implies $(x_0,t_0)$ is a regular point. Finally, we claim $\reg \mathcal{M}_t \cap \mathbb{H}$ and $\reg \mathcal{M}^*_t \cap \mathbb{H}$ are disjoint for $t$ sufficiently close to $t_0$. Suppose not, then there must be a first time of contact:
	\begin{align}
		t_1 = \inf\{t \mid  (\mathcal{M}^{s_0-}_{t}) \cap (\mathcal{M}^{s_0+})^*_t \cap \mathbb{H} \ne \emptyset \}
	\end{align}
	in $\mathbb{H}$. Given any point
	\begin{align*}
		(x_1,t_1) \in (\mathcal{M}^{s_0-}_t) \cap (\mathcal{M}^{s_0+})^*_t \subset \mathbb{H}.
	\end{align*}
	By definition of $s_0$ we know $(\mathcal{M}^{s_0+})^*_{t_1}$ is in fact a smooth MCF, so maximum principle \cref{maximum-principles} implies that $\mathcal{M}^{s_0-}$ agrees with $(\mathcal{M}^{s_0+})^*$ in some parabolic cylinder around $(x_1,t_1)$. The same reasoning applied to any other point in $\mathcal{M}^{s_0-}\cap (\mathcal{M}^{s_0+})^*$ shows that a connected component of $\mathcal{M}^{s_0-}$ agrees with a connected component of $(\mathcal{M}^{s_0+})^*$. This implies that $(\mathcal{M}^{s_0+})^*$ is simultaneously asymptotic to $\mathcal{C} \times [0,\infty)$ and $(\mathcal{C} + 2s_0e_{n+1}) \times [0,\infty)$, a contradiction. Hence the last condition in order to apply \cref{hopf-lemma} is satisfied and we conclude that $\mathcal{M} \cap \{x_{n+1} = s_0\}$ is regular. \par 
	Lastly we show that $s_0 = s$. This is a consequence of the fact that $\mathcal{M} \cap \Pi^{s_0}$ is compact. Using small balls as barriers similar to \cref{distance-estimate}, one sees that there exists some constant $N_1$ such that
	\begin{align*}
		\mathcal{M}_t \setminus B_{N_1R\sqrt{t+1}}(0) \subset \mathcal{T}_{R^{-1}\sqrt{t+1}}(\mathcal{C})
	\end{align*}
	for $R > 1$. On the other hand, for a fixed $t$ there is a constant $N_2$ such that 
	\begin{align*}
		\mathcal{M}_t \cap B_{N_1\sqrt{t+1}}(0) \subset \mathcal{T}_{N_2}(\mathcal{C}).
	\end{align*}
	as the first set is clearly compact. These two facts together imply the existence of a constant $N_3$ such that
	\begin{align*}
		\mathcal{M}_t \cap \{x_{n+1} = s_0\} \subset \mathcal{T}_{N_3}(\mathcal{C}) \cap \{x_{n+1} = s_0\}.
	\end{align*}
	This shows that $\mathcal{M}_t \cap \{x_{n+1} = s_0\}$ is compact (as $\mathcal{C} \cap \{x_{n+1} = s_0\}$ is compact), and since $t \in [0,T_0]$, $\mathcal{M} \cap \Pi^{s_0}$ is compact as well. To finish the proof, note that by the previous paragraph $\mathcal{M} \cap \{x_{n+1} = s_0\}$ consist of regular points only. At each regular point $(x_0,t_0)$ there is some $r = r(x_0,t_0)$ such that $\mathcal{M}$ is smooth in $P((x_0,t_0),r)$. Since $\mathcal{M} \cap \{x_{n+1} = s_0\}$ is compact, $r$ is uniformly bounded below away from 0, and this is a contradiction unless $s_0 = s$.
\end{proof}

\begin{proof}[Proof of \cref{reflection-symmetry}]
	To finish the proof we must show $S$ is nonempty, $S$ is open, and $S$ is closed. \par 
	Again by \cref{asymptotic-conical}, for sufficiently large $s$ we can make sure that $\mathcal{M}^{s+}$ is a smooth MCF,
	\begin{align*}
		(\mathcal{M}^{s+})^* \cap ((\mathcal{M}^{s-}) \cap \{x_{n+1} \ge 0\}) = \emptyset,
	\end{align*} 
	and that for any $(x,s_1,t) \in (\mathcal{M}^{s+})^*$ and $(x,s_2,t) \in \mathcal{M}^{0-}$ it holds that $s_1 - s_2 \ge 2s-1$. These two facts imply that for sufficiently large $s$ the inequality $(\mathcal{M}^{s+})^* > \mathcal{M}^{s-}$ is valid. On the other hand, by \cref{graphicality}, there is $R = R(T_0)$ such that $(\mathcal{M}^{0+})_{T_0} \setminus B_R(0)$ is graphical over $(\Pi^0)_{T_0}$. So for $s > R$ we have $(\mathcal{M}^{s+})_t$ is graphical over $(\Pi^s)_t$ for all $t \in [0,T_0)$. This shows $S$ is not empty. \par
	It is clear that $(\mathcal{M}^{s+})^* > \mathcal{M}^{s-}$ is an open condition. To see that the graphicality condition is also an open condition, let $\theta_t(x)$ be the angle between the unit normal to the approximate tangent space at a point $x \in \mathcal{M}_t$ and $e_{n+1}$. Suppose that $s \in S$, then graphicality is equivalent to $\theta_t(x) < \frac{\pi}{2}$ for all $t \in [0,T_0)$ and $x \in (\mathcal{M}^{s+})_t$. Since the flow $\mathcal{M}$ is $C^{2,\alpha}$-asymptotically conical, for given $t$ there exists $\varepsilon > 0$ such that $\theta_t(x) < \pi/2$ for all $x \in (\mathcal{M}^{s'+})_t$ where $\abs{s' - s} < \varepsilon$. Since the time interval is compact, there is a universal $\varepsilon$ such that the above holds for all $t \in [0,T_0)$. This shows openness of $S$. \par 
	Finally we show $S$ is closed. Obviously if $s \in S$ then $[s,\infty) \subset S$. So we assume $(s,\infty) \subset S$ and suppose for a contradiction that $s \not \in S$. At level $s$, either $(\mathcal{M}^{s+})^* \cap \mathcal{M}^{s-} \ne \emptyset$ or there is some $t_0 \in [0,T_0)$ such that $(\mathcal{M}^{s+})_{t_0}$ fails to be graphical over $(\Pi^s)_{t_0}$. \par 
	In the first case, $s$ is necessarily the first level of contact. By choosing $r$ small enough we can ensure $(\mathcal{M}^{s+})^*$ and $\mathcal{M}^{s-}$ are graphical in $P(X,r)$ where $X \in (\mathcal{M}^{s+})^* \cap \mathcal{M}^{s-}$. Moreover, by \cref{smoothness}, the reflected part $(\mathcal{M}^{s+})^*$ is a smooth MCF, so all the conditions of the maximum principle \cref{maximum-principles} are satisfied (note that the Gaussian density bound is automatic from the entropy bound). Applying \cref{maximum-principles}, we see $(\mathcal{M}^{s+})^*$ and $\mathcal{M}^{s-}$ agree in a neighborhood of $X$. Now an identical argument as in the proof of \cref{smoothness} shows that $(\mathcal{M}^{s_0+})^*$ is simultaneously asymptotic to $\mathcal{C} \times [0,\infty)$ and $(\mathcal{C} + 2s_0e_{n+1}) \times [0,\infty)$, which is again a contradiction. \par 
	In the second case, WLOG we may assume $t_0$ is the first time the graphicality condition fails. Then there necessarily exists a point $X = (x,s,t_0) \in \mathcal{M}_{t_0} \cap \{x_{n+1} = s\}$ whose tangent space contains the vector $e_{n+1}$. We again check the condition to apply Hopf lemma \cref{hopf-lemma} to $\mathcal{M}^1 = (\mathcal{M}^{s+})^*$, $\mathcal{M}^2 = \mathcal{M}^{s-}$ and $\mathbb{H} = \{x_{n+1} < s\}$ as in the proof of \cref{smoothness}. Tameness follows from \cref{tameness}. Since $e_{n+1}$ is normal to the hyperplane $\{x_{n+1} = s\}$ and $X$ is a regular point of $\mathcal{M}$ by \cref{smoothness}, we see that $\partial \mathbb{H}$ is not the tangent flow to either $\mathcal{M}^1$ or $\mathcal{M}^2$ (here we used the fact that the tangent flow at a regular point agrees with the static flow of the tangent plane). The disjointness of the regular parts of $\mathcal{M}^1$ and $\mathcal{M}^2$ in $\mathbb{H}$ follows identically as in the proof of \cref{smoothness}. Hence, we may apply \cref{hopf-lemma} to conclude that $\mathcal{M}^1$ and $\mathcal{M}^2$ have distinct tangents, which is a contradiction since the tangent spaces agree at $X$. This concludes the proof that $S$ is closed. \par 
	This shows that $S = (0,\infty)$. At $s = 0$, one sees that the graphicality condition is preserved (alternatively one can run the moving plane method from the other side, i.e. $s < 0$), but the strict inequality $(\mathcal{M}^{0+})^* > \mathcal{M}^{0-}$ does not hold anymore, which implies that $(\mathcal{M}^{0+})^* \cap  \mathcal{M}^{0-} \ne \emptyset$. Applying the maximum principle \cref{maximum-principles} once again we conclude $(\mathcal{M}^{0+})^* = \mathcal{M}^{0-}$, and this is the required reflection symmetry across $\Pi^0 = \{x_{n+1} = 0\}$. 
\end{proof}

 \section{Rotational symmetry}
\label{rotational-symmetry-section}
In this section we prove \cref{main-theorem} and \cref{simons-cone-thm}. First let $\mathcal{C}$ be a smooth, rotationally symmetric double cone. A typical example of such a $\mathcal{C}$ is given by \cref{good-cone} which has the $x_1$ axis as its axis of symmetry. Our theorem allows more general cones of the form 
\begin{align*}
	x_1^2 =  \begin{cases} m_1^2(x_2^2 + x_3^2 + \cdots + x_n^2), & x_1 \ge 0 \\  m_2^2(x_2^2 + x_3^2 + \cdots + x_n^2), & x_1 < 0\end{cases},
\end{align*}
where $m_1, m_2 >0$ (i.e. the top and bottom parts of the cone can have different cone angles). A cone $\mathcal{C}$ of the above form should also satisfy $\lambda[\mathcal{C}] < 2$, but this is not explicitly known. See \cref{further-remarks} for more on the entropy of cones. 
\begin{proof}[Proof of \cref{main-theorem}]
	WLOG we may assume the axis of symmetry is the $x_1$-axis. Observe that rotational symmetry is equivalent to reflection symmetry across every hyperplane containing the $x_1$-axis. Up to a rotation it suffices to show that $\mathcal{M}$ is symmetric across the hyperplane $\{x_{n+1} = 0\}$. Since $\mathcal{C}$ is a smooth graph over $\{x_{n+1} = 0\}$, the desired conclusion follows from \cref{reflection-symmetry}. 
\end{proof}

\begin{proof}[Proof of \cref{cylindrical-singularity}]
	The rotational symmetry is \cref{main-theorem}. Since $\mathcal{C}$ is symmetric across the hyperplane $\{x_1 = 0\}$ as well, we can apply \cref{reflection-symmetry} to $\mathbb{H} = \{x_1 > 0\}$ to conclude that $\mathcal{M}$ is smooth away from $\{x_1 = 0\}$. Together with \cref{main-theorem} we infer that the only possible singularity of $\mathcal{M}$ is at the origin. Moreover, any tangent flow $\mathcal{X}$ at the first singular time must be rotationally symmetric. By the classification of rotationally symmetric self-shrinkers of Kleene and M\o ller \cite{KleeneMoller}, $\mathcal{X}$ has to be one of the following: a round sphere, a round cylinder $\mathbb{R} \times \mathbb{S}^{n-1}$ or a smooth embedded $\mathbb{S}^1 \times \mathbb{S}^{n-1}$. Since $\mathcal{M}$ is not closed, we conclude that $\mathcal{X}$ has to be a round cylinder. The uniqueness of the cylinder follows from the work of Colding--Minicozzi \cite{CMUniqueness}.
\end{proof}
\begin{rem}
	The above corollary does not guarantee the existence of a cylindrical singularity, as it is entirely possible that the flow remains smooth for all times. See \cref{disconnection} for sufficient conditions for the flow to have a singularity.
\end{rem}
Next we apply the same method to cones with more general symmetry groups. For this part we will work, for convenience, in $\mathbb{R}^{n+2}$ instead. Let $O(p)$ denote the symmetry group of $\mathbb{S}^{p-1} \subset \mathbb{R}^p$. Fix an integer $1 \le p \le n-1$ and suppose $\mathcal{C}$ is a smooth double cone with $\lambda[\mathcal{C}] < 2$ that has symmetry group $O(p+1) \times O(n-p+1)$. Typical examples are cones $\mathcal{C}_{n,p}$ over the families of minimal hypersurfaces in $\mathbb{S}^n$ given by
\begin{align*}
	\mathcal{S}_{n,p} = \sqrt{\frac{p}{n}} \mathbb{S}^p \times \sqrt{\frac{n-p}{n}} \mathbb{S}^{n-p} \subset \mathbb{S}^n.
\end{align*}
The cones $\mathcal{C}_{n,p}$ are known as Simons-type cones. An immediate consequence of \cref{simons-cone-thm} is:
\begin{cor}
	Any smooth self-expander coming out of a Simons-type cone $\mathcal{C}_{n,p}$ inherits the $O(p+1) \times O(n-p+1)$-symmetry of $\mathcal{C}_{n,p}$.
\end{cor}
\begin{rem}
	Similar results for minimal surfaces have been obtained by Mazet \cite{Mazet}, using the elliptic moving plane method.
\end{rem}
\begin{proof}[Proof of \cref{simons-cone-thm}]
	Write $(x_1,\ldots,x_{p+1},y_1,\ldots,y_{n-p+1})$ the standard coordinates on $\mathbb{R}^{n+2}$.  WLOG we may assume $\mathcal{L}(\mathcal{C)} = \sigma_1 \times \sigma_2$ where $\sigma_1$ is rotationally symmetric across the $x_1$ axis and $\sigma_2$ rotationally symmetric across the the $y_1$-axis. Evidently showing the rotational symmetry in $x$-coordinates is enough, as the identical argument works for the $y$-coordinates. It suffices to show that the reflection symmetry is preserved through all hyperplanes of the form
	\begin{align*}
		\sum_{i = 2}^{p+1} c_i x_i = 0,
	\end{align*}
	which, up to an ambient rotation in the $x$-coordinates, we may assume to be $\{x_{p+1} = 0\}$. Note that the a cone with $O(p+1)\times O(n-p+1)$ symmetry takes the form (up to relabeling) 
	\begin{align*}
		\sum_{i=1}^{p+1} a_i x_i^2 = \sum_{j=1}^{n-p+1} b_jy_j^2
	\end{align*}
	for suitable choices of coefficients $a_i$ and $b_j$. It is not hard to see that $\mathcal{C} \cap \{x_{p+1} > 0\}$ is a graph over $\{x_{p+1} = 0\}$ via
	\begin{align*}
		x_{p+1} = a_{p+1}^{-1/2}\left(\sum_{j=1}^{n-p+1} b_jy_j^2 - \sum_{i=1}^p a_ix_i^2\right)^{1/2}.
	\end{align*}
	Hence \cref{reflection-symmetry} applies and the desired reflection symmetry follows (of course, we have to put
	\begin{align*}
		\Pi^s = \{(x,x_{p+1},y) \in \mathbb{R}^{n+2} \mid x_{p+1} = s\} \times [0,\infty) \subset \mathbb{R}^{n+2} \times [0,\infty).
	\end{align*}
	and change the dimensions in the proofs accordingly).
\end{proof}

\section{Construction of the matching motion}
\label{construction-of-matching-motion}
In this section we use the ideas of Bernstein--Wang in \cite{BWTopologicalUniqueness} and \cite{BWTopology} to produce an immortal Brakke flow starting from a cone $\mathcal{C}$ with $\lambda[\mathcal{C}] < 2$ that is not self-expanding. This demonstrates that \cref{main-theorem} is not void. \par 
\subsection{Self-expanders}
\label{expander-preliminary}
Let us briefly summarize some basic facts about self-expanders. It is often helpful to consider the variational characterization of self-expanders. Formally, self-expanders are critical points of the functional:
\begin{align*}
	E[\Sigma] = \int_\Sigma e^{\frac{\abs{x}^2}{4}} d\mathcal{H}^n
\end{align*}
and equation \cref{self-expander-equation} corresponds to the Euler-Lagrange equation of $E[\Sigma]$. We record the second variation of $E$. For a proof see for example Proposition 4.2 in \cite{BWSpace}. 
\begin{thm}
	If $\{\phi_t\}_{t \in (-\varepsilon,\varepsilon)}$ is a compactly supported normal variation of $\Sigma$ with $\left.\frac{d\phi_t}{dt} \right|_{t = 0} = f\nu_\Sigma$, where $\nu_\Sigma$ is the outwards unit normal. Then 
	\begin{align*}
		\left.\frac{d^2}{dt^2}\right|_{t=0} E[\phi_t(\Sigma)] = -\int_{\Sigma} f L_\Sigma f d\mathcal{H}^n
	\end{align*}
	where  $L_\Sigma$ is the stability operator of $\Sigma$ given by
	\begin{align*}
		L_\Sigma = \Delta_\Sigma + \frac{1}{2} x\cdot \nabla_\Sigma + \abs{A_\Sigma}^2 - \frac{1}{2}.
	\end{align*}
\end{thm}
A real number $\mu \in \mathbb{R}$ is an eigenvalue for $-L_\Sigma$ if there is a function $u \in W^1_{\frac{1}{4}}(\Sigma) \setminus \{0\}$ such that $-L_\Sigma u = \mu u$, where 
\begin{align*}
	W^1_{\frac{1}{4}}(\Sigma) = \{f: \Sigma \to \mathbb{R} \mid \int_\Sigma (\abs{f}^2 + \abs{\nabla f}^2)e^{\frac{\abs{x}^2}{4}} d\mathcal{H}^n < \infty\}.
\end{align*}
The \textit{index} of a self-expander is the number of negative eigenvalues of $-L_\Sigma$, which is equal to
\begin{align*}
	\sup \{\dim V \mid V \text{ linear subspace } \subset C_0^2(\Sigma), -\int_{\Sigma} f L_\Sigma f \le 0 \; \forall f \in V \setminus \{0\}\}.
\end{align*}
We say a self-expander is \textit{stable} if it has index zero. By Lemma 4.1 in \cite{BWIntegerDegree}, the operator $L_\Sigma$ is formally adjoint in $W^0_{\frac{1}{4}}(\Sigma)$, and $L_\Sigma$ has a discrete spectrum; that is, the eigenvalues of $-L_\Sigma$ can be ordered as $\mu_1 < \mu_2 < \cdots < \mu_n < \cdots$. Moreover, the space of eigenfunctions associated to the lowest eigenvalue $\mu_1$ is 1-dimensional, and any eigenfunction $f$ of $\mu_1$ has a sign. \par 
We need the following basic distance estimates of asymptotically conical self-expanders.
\begin{prop}
	\label{expander-curvature-estimates}
	Let $\mathcal{C}$ be a smooth cone. Suppose $\Sigma$ is a self-expander $C^{2,\alpha}$-asymptotic to $\mathcal{C}$, then there is $N > 0$ such that $\Sigma \setminus B_{NR}(0) \subset \mathcal{T}_{R^{-1}}(\mathcal{C})$ for $R > 1$.
\end{prop}
\begin{proof}
	Since $\Sigma$ is smooth and $\rho \Sigma \to \mathcal{C}$ in $C^{2,\alpha}_{loc}(\mathbb{R}^n \setminus \{0\})$, for sufficiently small $\rho$ we have on $\rho\Sigma \cap (B_{1}(0) \setminus \{0\})$
	\begin{align*}
		\abs{A_{\rho\Sigma}} \le C.
	\end{align*}
	Hence for $x \in \Sigma$ with $\abs{x}$ sufficiently large depending on the above,
	\begin{align*}
		\abs{A(x)} = \rho \abs{A_{\rho \Sigma}(\rho x)} \le C\abs{x}^{-1}
	\end{align*}
	if we pick $\rho = \frac{1}{2}\abs{x}^{-1}$ so that $\rho x \in B_1(0)$. This proves that $\abs{A(x)} \le C \abs{x}^{-1}$ for all $x \in \Sigma$. 
	Together with the self-expander equation, these imply that there is $C > 0$ with
	\begin{align*}
		\dist(x, \mathcal{C}) < C \abs{x}^{-1}
	\end{align*}
	for $x \in \mathcal{C} \setminus B_1(0)$. Finally by scaling it follows that there is $N$ such that for $R \ge 1$.
	\begin{align*}
		\Sigma \setminus B_{NR}(0) \subset \mathcal{T}_{R^{-1}}(\mathcal{C}). &\qedhere
	\end{align*}
\end{proof}
\begin{rem}
	Note that similar to the above we can also estimate the derivatives of $A$: 
	\begin{align*}
		\abs{\nabla^m A} \le C \abs{x}^{-m-1},
	\end{align*}
	provided the cone is sufficiently regular.
\end{rem}

\subsection{Mean curvature flow with boundary}
\label{prelim-mcf-boundary}
Since we are dealing with noncompact initial hypersurfaces, many existence theorems (in particular the unit density theorem 11.4 in \cite{Ilmanen}) do not apply directly in our case. To account for this we will utilize White's recent work on MCF with boundary \cite{WhiteMCFBoundary}, which is a generalization of the Brakke flow in \cref{brakkeflow}. For simplicity we will only work in the ambient manifold $\overbar{B_R(0)}$. \par 
Given a hypersurface $\Sigma$ with boundary $\Gamma \subset \partial B_R(0)$ in an open set $U \subset \overbar{B_R(0)}$ and an integral $n$-rectifiable Radon measure $\mu$, the first variation formula becomes:
\begin{align*}
	\int \Div_{V(\mu)} X d\mu = -\int H \cdot X d\mu + \int \nu_\mu \cdot X d(\mathcal{H}^{n-1} \llcorner \Gamma)
\end{align*}
for any compactly supported $C^1$ vector field $X$, where $H$ is the generalized mean curvature vector and $\nu$ the approximating normal vector to $\Gamma$. By an \textit{integral $n$-Brakke flow with boundary $\Gamma$} in $U$ we mean a family of integral $n$-rectifiable Radon measures $\mathcal{M} = \{\mu_t\}_{t \in I}$ satisfying the items (a), (b) and (c) in \cref{brakkeflow} with the extra condition:
\begin{enumerate}[label = (\alph*)]
	\setcounter{enumi}{3}
	\item For a.e. $t \in I$ the normal vector satisfies $\abs{\nu_{\mu_t}} \le 1$ $\mathcal{H}^{n-1}$-a.e on $\Gamma$.
\end{enumerate}
For simplicity we will refer to the above as Brakke flow with boundary $\Gamma$. This is unambiguous since, by item (d), the boundary $\Gamma$ stays unchanged under the Brakke flow. \par 
As before, a Brakke flow with boundary $\Gamma$ is unit-regular if, for a spacetime point $X = (x,t)$, $\mathcal{M}$ is smooth and has no sudden mass loss if a tangent flow at $X$ is a multiplicity one plane. $\mathcal{M}$ is cyclic if the associated mod-2 flat chain $[V(\mu_t)]$ has boundary equal to $\Gamma$. By works of White \cite{WhiteMCFBoundary}, Brakke flows with boundary $\Gamma$ produced by elliptic regularization are unit-regular and cyclic.
\begin{thm}[Theorem 1.1, Theorem 14.1 of \cite{WhiteMCFBoundary}]
	\label{brakkeflow-with-boundary}
	Let $\Sigma \subset \overbar{B_R(0)}$ be a hypersurface with boundary $\Gamma \subset \partial B_R(0)$. There exists a unit-regular and cyclic Brakke flow with boundary $\Gamma$, $\mathcal{M} = \{\mu_t\}_{t \in [0,\infty)}$ with $\mu_0 = \mathcal{H}^{n} \llcorner \Sigma$.
\end{thm}
Similarly Brakke flow with boundary needs not be unique, but White's theorem says that a unit-regular and cyclic one always exists. White proved in addition a strong boundary regularity theorem (Theorem 17.1 of \cite{WhiteMCFBoundary}) in the codimension one case, ruling out a scenario where interior singularities could accumulate to a boundary singularity. Hence the boundary $\Gamma$ truly remains unchanged in the classical sense.  

\subsection{Level set flows and matching motions}
The final ingredient we need is the set-theoretic generalization of MCF, initially developed by \cite{ChenGigaGoto} and \cite{EvansSpruck} as viscosity solutions to certain PDEs. Given a closed set $\Gamma_0 \subset \mathbb{R}^{n+1}$, we choose any uniformly continuous function $u_0$ such that $\Gamma_0 = \{x \in \mathbb{R}^{n+1}\mid u_0(x) = 0\}$. There exists a unique $u \in C(\mathbb{R}^{n+1} \times [0,\infty))$ which is a viscosity solution to the problem
\begin{align*}
	\begin{cases}
		u_t = \sum_{i,j=1}^{n+1} \left(\delta_{ij} - \frac{u_{x_i}u_{x_j}}{\abs{\nabla u}^2}\right)u_{x_ix_j} & \text{ on } \mathbb{R}^{n+1} \times [0,\infty) \\
		u(x,0) = u_0(x) & \text{ on } \mathbb{R}^{n+1} \times \{0\}.
	\end{cases}
\end{align*}
Let $\Gamma_t = \{x \in \mathbb{R}^{n+1} \mid u(x,t) = 0\}$. We call $\mathcal{K} = \bigcup_{t \in [0,\infty)} \Gamma_t \times \{t\}$ the \textit{level set flow} of $\Gamma_0$.  \par 
Since the viscosity solution is unique, level set flow is also unique with given initial data. However, level set flows might fatten, i.e. $\mathcal{K}$ might develop a non-empty interior (for example the figure eight fattens immediately). Formally, the level set flow of $\Gamma_0$ fattens if $\mathcal{H}^{n+1}(\Gamma_t) > 0$ for some $t > 0$. A theorem of Ilmanen (11.3 in \cite{Ilmanen}) shows that fattening phenomenon is not generic and can therefore be perturbed away. \par 
Alternatively, it has been observed that the level set flow can be characterized as the "biggest flow" of a closed set satisfying the avoidance principle. There is a rich literature on this more geometrically intuitive way of handling set flows and we refer to \cite{WhiteTopology}, \cite{HershkovitsWhite} and \cite{BCW} for more information on this approach. \par

Ilmanen \cite{Ilmanen} combined ideas from Brakke flows and level set flows and introduced the notion of a matching motion, which turns out to be the suitable notion for our purposes. Let $I_n(U)$ be the set of integral $n$-current in $U$.
\begin{defn}[8.1, 9.1 of Ilmanen \cite{Ilmanen}]
	Let $\mathcal{K} \in I_{n+1}(\mathbb{R}^{n+1} \times \mathbb{R}^+)$, $\mathcal{M} = \{\mu_t\}_{t \in [0,\infty)}$ be a Brakke flow and $\Gamma_0 \in I_n(\mathbb{R}^{n+1})$ with finite mass and empty boundary. A pair $(\mathcal{K},\mathcal{M})$ is an \textit{enhanced motion} with initial data $\Gamma_0$ if
	\begin{enumerate}[label = (\alph*)]
		\item $\partial \mathcal{K} = \Gamma_0$ and $\mathcal{K}_t \in I_n(\mathbb{R}^{n+1})$ for a.e. $t \ge 0$;
		\item $\partial \mathcal{K}_t = 0$ and $t \to \mathcal{K}_t$ is continuous in the flat topology for $t \ge 0$;
		\item $\mu_0 = \mu_{\Gamma_0}$, $\mathbb{M}[\mu_t] \le \mathbb{M}[\mu_0]$ and $\mu_{\mathcal{K}_t} \le \mu_t$ for a.e. $t \ge 0$.
	\end{enumerate}
	If the pair $(\mathcal{K},\mathcal{M})$ further satisfies
	\begin{enumerate}[label = (\alph*),resume]
		\item $\mu_t = \mu_{\mathcal{K}_t} = \mu_{V(\mu_t)}$ for $t \ge 0$,
	\end{enumerate}
	then we call it a \textit{matching motion}.
\end{defn}
Note that in the above definition we have already abused some notation. Indeed, in our applications $\mathcal{K}$ is going to be the level set flow from $\Gamma_0$. A fundamental result of Ilmanen (Section 12 of \cite{Ilmanen}) shows that a nonfattening level set flow is a matching motion, which justifies our abuse of notation here. \par 
We will use the following result of S. Wang \cite{Shengwen} which asserts that limit of low entropy matching motions is a matching motion. Recall that a sequence of Brakke flow $\mathcal{M}_i = \{\mu_t^i\}_{t \in [0,\infty)}$ converges to $\mathcal{M} = \{\mu_t\}_{t \in [0,\infty)}$ if $\mu_t^i \to \mu_t$ as Radon measures and, after possibly passing to a subsequence, $V(\mu_t^i) \to V(\mu_t)$ as varifolds for a.e $t \in [0,\infty)$. 
\begin{prop}[Theorem 3.5 of \cite{Shengwen}]
	\label{limit-matching-motion}
	Let $(\mathcal{K}_i,\mathcal{M}_i)$ be a sequence of matching motions converging to an enhanced motion $(\mathcal{K},\mathcal{M})$ with $\lambda[\mathcal{M}] < 2$, then $(\mathcal{K},\mathcal{M})$ is a matching motion.
\end{prop}
\begin{rem}
	The theorem fails without the entropy assumption as the set-theoretic limit of a sequence of grim reapers (which has entropy 2) is two lines but the limit in the sense of currents is empty (as the two lines cancel each other). 
\end{rem} 

\subsection{Construction of the smooth flow}
\label{construction-of-smooth-flow}
Before we construct the weak flow, we briefly summarize the construction of Morse flow line starting from an unstable self-expander \cite{BWTopologicalUniqueness}. For reader's convenience we have also included a summary of the theorems with some proofs in \cref{smooth-construction-appendix}. \par 
Recall that for an unstable self-expander $\Sigma$, the stability operator $L_\Sigma$ has discrete spectrum and eigenfunctions to the lowest eigenvalue $\mu_1 < 0$ have a sign. Let $f$ be the unique positive eigenfunction of $\mu_1$ with $\norm{f}_{W^0_{\frac{1}{4}}(\Sigma)} = 1$. For $\varepsilon > 0$ we form the perturbations of $\Sigma$ by $f$ given by
\begin{align}
\label{sigma-epsilon}
	\Sigma^\varepsilon = \Psi^\varepsilon(\Sigma) \text{ where } \Psi^\varepsilon(x) = x + \varepsilon f(x) \nu_\Sigma.
\end{align}
By \cref{eigenfunction-decay-estimates}, there is $N> 0$ depending on $\varepsilon$ such that $\Sigma^\varepsilon \setminus B_{NR}(0) \subset \mathcal{T}_{R^{-1}}(\mathcal{C})$ for all $R > 1$. By \cref{expander-mean-convexity}, $\Sigma^\varepsilon$ is expander mean-convex (see \cref{smooth-construction-appendix} for the definition) for $\varepsilon$ sufficiently small. By the existence theorem \cref{existence-theorem} there is a unique MCF starting from an asymptotically conical, expander mean-convex hypersurface, and the expander-mean-convexity is preserved along the flow until the first singular time. Applying \cref{existence-theorem} to $\Sigma^\varepsilon$, we get for each $\varepsilon$ a unique MCF $\mathcal{M}^\varepsilon = \{\Sigma^\varepsilon_t\}_{t \in [1,T^\varepsilon)}$ with $\Sigma^\varepsilon_1 = \Sigma^\varepsilon$. Moreover, by \cref{eigenfunction-decay-estimates} and \cref{expander-curvature-estimates}, $\Sigma^\varepsilon$ has uniformly bounded curvature, so the interior estimates of Ecker--Huisken \cite{EHInterior} implies that the interval of existence of is independent of $\varepsilon$. Moreover, at first singular time $T^\varepsilon$, 
\begin{align*}
	\lim_{t \to T^\varepsilon} \sup_{\Sigma^\varepsilon_t \cap B_{N'\sqrt{t}}} \abs{A_{\Sigma_t^\varepsilon}} = \infty
\end{align*}
for some constant $N' > 0$. \par 

\subsection{Construction of the Brakke flow}
We now turn to our construction of the weak flow. One of the method to construct a weak flow is via capping off the hypersurfaces $\Sigma^\varepsilon \cap \overbar{B_R(0)}$ smoothly and taking a sequence of Brakke flows (or weak set flows) starting from these capped-off hypersurfaces. However, one has to be careful with the above method when working with entropy bounds, since the cap might increase the entropy. On the other hand, the capping method has the advantage that, by a suitable choice of the caps, the expander mean-convexity is preserved through the cap and hence in the limits (here with weak set flows one has to interpret the mean-convexity in the weak sense as well). Such a construction for self-shrinkers can be found in Section 7 of \cite{CCMSGeneric} and, for self-expanders, in \cite{BCW}. \par 
Here we use an alternative approach of Brakke flow with boundary. This construction is less technical and respects the entropy well. However, it is also less clear how expander mean-convexity is preserved through the flows and hence in the limit. This is not needed in our case because we are only concerned with existence. It is an interesting question to determine whether the flows constructed from the two methods above agree (they are, of course, the same notion when smooth, so the point is to determine how each of them flows past singularities). We believe this should be the case with the entropy bound $\lambda[\mathcal{C}] < 2$, but the general picture might be less clear. \par 
The following proposition is the weak flow analogy of the smooth flow produced in \cref{existence-theorem}.
\begin{prop}
	\label{prop32}
	Let $\mathcal{C}, \Sigma$ be as in \cref{nontrivial-flow-lines}, and $\Sigma^\varepsilon$ be as in \cref{sigma-epsilon}. There exists $\varepsilon_0 > 0$ such that, for $\abs{\varepsilon} < \varepsilon_0$, there exists an immortal matching motion $(\mathcal{K}^\varepsilon,\mathcal{M}^\varepsilon)$ where $\mathcal{K}^\varepsilon = \{\Gamma^\varepsilon_t\}_{t \in [1,\infty)}$ and $\mathcal{M}^\varepsilon = \{\mu^\varepsilon_t\}_{t \in [1,\infty)}$ such that $\Gamma^\varepsilon_1 = \Sigma^\varepsilon$ and $\mu^\varepsilon_1 = \mathcal{H}^{n} \llcorner \Sigma^\varepsilon$. Moreover the flow $(\mathcal{K}^\varepsilon,\mathcal{M}^\varepsilon)$ agrees with the smooth flow starting from $\Sigma^\varepsilon$ for $t \in [1,T^\varepsilon)$.
\end{prop}
\begin{proof}
	Suppose $\varepsilon > 0$, as the argument for $\varepsilon < 0$ is identical. Let $\Sigma^{\varepsilon,R} = \Sigma^\varepsilon \cap \overbar{B_R(0)}$ be the hypersurface in $\overbar{B_R(0)}$ with boundary $\Sigma^{\varepsilon} \cap \partial B_R(0)$. By \cref{brakkeflow-with-boundary}, there exists an unit-regular and cyclic Brakke flow with boundary $\mathcal{M}^{\varepsilon,R} = \{\mu^{\varepsilon,R}_t\}_{t \in [0,\infty)}$ starting from $\Sigma^{\varepsilon,R}$. The flow $\mathcal{M}^{\varepsilon,R} \llcorner B_{R/2}(0)$ is therefore a (usual) Brakke flow inside $B_{R/2}(0)$. Since nonfattening is generic, we may choose a sequence $R_i \to \infty$ such that the associated level set flow of $\mathcal{M}^{\varepsilon,R_i} \llcorner B_{R_i/2}(0)$ is nonfattening. This produces a sequence of matching motions
	\begin{align*}
			(\mathcal{K}^{\varepsilon,R_i},\mathcal{M}^{\varepsilon,R_i} \llcorner B_{R_i/2}(0)).
	\end{align*}
	By compactness of Brakke flow we may now pass to a subsequence $R_i \to \infty$ to obtain a limiting enhanced motion $(\mathcal{K}^{\varepsilon},\mathcal{M}^\varepsilon)$ in $\mathbb{R}^{n+1}$ starting from $\Sigma^\varepsilon$. \par 
	By Lemma 3.5 of \cite{BWSmoothCompactness}, $\lambda[\Sigma] = \lambda[\mathcal{C}] < 2$, and by Lemma 6.2 of \cite{BWTopologicalUniqueness}, for every $\delta > 0$ there exists $\varepsilon_0$ such that $\abs{\lambda[\Sigma^\varepsilon] - \lambda[\Sigma]} < \delta$ for $\varepsilon < \varepsilon_0$. Choosing $\delta$ small enough so that $\lambda[\mathcal{C}] + \delta < 2$ and $\varepsilon_0$ small according to $\delta$ ensures that $\lambda[\mathcal{M}^\varepsilon] = \lambda[\Sigma^\varepsilon] < 2$, and so, in view of \cref{limit-matching-motion}, $(\mathcal{K}^\varepsilon,\mathcal{M}^\varepsilon)$ is matching. \par 
	Finally, using the argument in \cref{interior-estimates} with pseudolocality of MCF replaced by that of Brakke flow there exist $\delta > 0$ and $N' > 0$ such that $\supp \mu^\varepsilon_t \setminus B_{N'\sqrt{t}}(0) = \Sigma^\varepsilon_t \setminus B_{N'\sqrt{t}}(0)$ for $t \in [1,1+\delta]$. Since $(\mathcal{K}^\varepsilon,\mathcal{M}^\varepsilon)$ is matching, $\Gamma_t^\varepsilon \setminus B_{N'\sqrt{t}}(0) = \Sigma^\varepsilon_t \setminus B_{N'\sqrt{t}}(0)$ as well. It follows from uniqueness of level set flow that $\Gamma_t^\varepsilon$ agrees with $\Sigma_t^\varepsilon$. Using the matching property again we infer that $\supp \mu_t^\varepsilon = \Sigma_t^\varepsilon$. It is easy to see that these flow agree up to the first singular time of $\Sigma^\varepsilon_t$ (i.e. $T^\varepsilon$).
\end{proof}

Let $(\mathcal{K}^\varepsilon,\mathcal{M}^\varepsilon)$ be the matching motions constructed as above. We can once again take a limit as $\varepsilon \to 0^+$ to obtain a limiting enhanced motion $(\mathcal{K},\mathcal{M})$ where $\mathcal{K} = \{\Gamma_t\}_{t \in [1,\infty)}$ and $\mathcal{M} = \{\mu_t\}_{t \in [1,\infty)}$ such that $\Gamma_0 = \Sigma$ and $\mu_0 = \mathcal{H}^n \llcorner \Sigma$. However, this limit is not enough to prove \cref{nontrivial-flow-lines} as we will only recover the flow of the self-expander $\Sigma$ (as in the case of the smooth flow). Moreover, the limit does not attain the cone as its initial data. We need to translate the flow properly so that the starting time can be extended back to 0, and argue that we do not get the original flow of the self-expander back in the process. 
\begin{proof}[Proof of \cref{nontrivial-flow-lines}] Again WLOG suppose $\varepsilon > 0$. Let $(\mathcal{K}^\varepsilon,\mathcal{M}^\varepsilon)$ be the matching motions from \cref{prop32}. It is convenient to work with the following rescaled MCF: 
	\begin{align*}
		\tilde{\mu}^{\varepsilon}_s = \mu^{\varepsilon}_t \circ t^{-\frac{1}{2}} \text{ and } \tilde{\Gamma}^{\varepsilon}_s = \Gamma^\varepsilon_t, \;\; s = \log t.
	\end{align*}
	and let $(\tilde{\mathcal{K}}^\varepsilon, \tilde{\mathcal{M}}^\varepsilon)$ denote the rescaled flow. Under the rescaling, a smooth MCF will satisfy the rescaled MCF equation
	\begin{align*}
		\left(\frac{\partial x}{\partial s}\right)^\perp = H_{\Sigma_s} - \frac{x^\perp}{2},
	\end{align*}
	which has self-expanders as stable solutions. The flow $\tilde{\mathcal{M}}^\varepsilon$ is defined on the time interval $[0,\infty)$ with $\supp \tilde{\mu}^\varepsilon_0 = \Sigma^\varepsilon$. Since $\Sigma$ is unstable, the lowest eigenvalue $\lambda_1$ of $-L_\Sigma$ satisfies $\lambda_1 < \frac{1}{2}$, and consequently the $\mathcal{M}^{\varepsilon}$ "flows faster" than the parabolic rescaling $\sqrt{t}$. To be precise, for the the flow of the self-expander $\Sigma$ we have 
	\begin{align*}
		\lim_{\lambda \to 0^+} \dist(\Sigma, \lambda \Sigma_{\lambda^{-2}}) = \lim_{\lambda \to 0^+} \dist(\Sigma, \lambda \sqrt{\lambda^{-2}} \Sigma) = 0
	\end{align*}
	but since we are perturbing $\Sigma$ by its eigenfunction whose eigenvalue is below $\frac{1}{2}$, we must have 
	\begin{align*}
		\lim_{\lambda \to 0^+} \dist(\Sigma^\varepsilon, \lambda \supp \mu^\varepsilon_{\lambda^{-2}}) = \infty.
	\end{align*}
	In the rescaled setting, the above means that $\tilde{M}^\varepsilon$ moves out exponentially (the exact rate of which depends on $\lambda_1$). As such, we can find a sequence of time $s_\varepsilon$ with
	\begin{align*}
		d(\supp \tilde{\mu}^\varepsilon_{s_\varepsilon},x_0) = \gamma
	\end{align*}
	for some fixed point $x_0 \in \Sigma$ and positive constant $\gamma$. On the other hand $\varepsilon \to 0^+$, the rescaled flows $\tilde{\mathcal{M}}^\varepsilon$ converge to the static rescaled flow of $\Sigma$, so in fact $s_\varepsilon \to \infty$ as $\varepsilon \to 0^+$, i.e. one has to go further in time to reach a distance $\gamma$ away from $x_0 \in \Sigma$. This fact allows us to time translate $\tilde{\mathcal{M}}$ by $-s_\varepsilon$ to obtain a sequence of rescaled MCFs $\tilde{\mathcal{M}}^{\varepsilon,s_\varepsilon} = \{\tilde{\mu}^{\varepsilon,s_\varepsilon}_s\}_{s \in [-s_\varepsilon,\infty)}$ such that 
	\begin{align*}
		\supp \tilde{\mu}^{\varepsilon,s_\varepsilon}_{-s_\varepsilon} = \Sigma^\varepsilon \text{ and } d( \supp \tilde{\mu}^{s,s_\varepsilon}_0, x_0) = \gamma.
	\end{align*}
	By compactness of Brakke flows we can take $\varepsilon \to 0^+$ to obtain a limiting flow $\tilde{\mathcal{M}}$ defined on $(-\infty,\infty)$. It is easy to see that $\tilde{\mathcal{M}}$ is not the flow of $\Sigma$ as $d(\supp \tilde{\mu}_0, \Sigma) \ge \gamma$. Moreover, since $\Sigma$ is asymptotically conical and $\Sigma^\varepsilon \to \Sigma$ smoothly as $\varepsilon \to 0$, it follows that 
	\begin{align*}
		\lim_{s \to -\infty} \supp \tilde{\mu}_s = \lim_{\varepsilon \to 0} \supp \tilde{\mu}_{-s_\varepsilon}^{\varepsilon,s_\varepsilon} = \lim_{\varepsilon \to 0} \Sigma^\varepsilon = \Sigma.
	\end{align*}
	In view of the rescaling, this proves that the flow achieves $\mathcal{C}$ as the initial data. \par 
	Since $\tilde{\mu}_s$ is integral, it follows from strong maximum principle for varifolds \cite{SolomonWhite} that $\lim_{s \to -\infty} \tilde{\mu}_s = k \mathcal{H}^n \llcorner \Sigma$. Since $\lambda[\mathcal{C}] < 2$, $k = 1$, so the desired regularity follows from Brakke regularity theorem.
\end{proof}
\subsection{An example of singularity formation}
To conclude the section, we give a sufficient condition when there exists a flow coming out of $\mathcal{C}$ that has a singularity. The proof makes heavy use of the structure theory of self-expanders developed by Bernstein and Wang in a series of papers starting from \cite{BWSpace}. It would be interesting to see if a simpler proof exists. 
\begin{prop}
\label{disconnection}
	Suppose $\mathcal{C} \subset \mathbb{R}^{n+1}$, $2 \le n \le 6$, is a smooth double cone with $\lambda[\mathcal{C}] < 2$ and that the two connected components of $\mathcal{L}(\mathcal{C})$ are graphs over some (fixed) hyperplane. Suppose that there is a connected self-expander asymptotic to $\mathcal{C}$.	Then there exists an integral Brakke flow coming out of $\mathcal{C}$ that has a singularity in finite time.
\end{prop}
\begin{rem}
	We note that the above is consistent with the topological uniqueness result of \cite{BWTopologicalUniqueness}. Indeed, by Proposition 5.6 of \cite{BWTopologicalUniqueness}, if $\lambda[\mathcal{C}] < \lambda[\mathbb{S}^{n-1} \times \mathbb{R}]$, the flow produced by \cref{nontrivial-flow-lines} is smooth for all time. On the other hand, any such double cone will not have a connected self-expander.
\end{rem}
\begin{proof}
	 Let $\sigma = \mathcal{L}(\mathcal{C})$, and let $W$ be the connected component of $\mathbb{S}^n$ lying between the two connected components of $\sigma$. By Corollary 1.2 of \cite{BWSpace}, the set of generic cones (in the sense that there is no $C^{2}$-asymptotically self-expander with nontrivial Jacobi fields that fix the infinity) whose link lie in $W$, is dense near $\mathcal{C}$.  These facts allow us to take a sequence of $C^{2,\alpha}$-hypersurfaces $\sigma_i$ in $\mathbb{S}^2$ such that 
	\begin{itemize}
		\item $\sigma_i \to \sigma$ in $C^{2,\alpha}(\mathbb{S}^{n})$ as $i \to \infty$;
		\item $\mathcal{C}_i$ is a generic, smooth double cone for all $i$, where $\mathcal{C}_i$ is the cone over $\sigma_i$;
		\item $\lambda[\mathcal{C}_i] < 2$ for sufficiently large $i$, by Lemma 6.2 of \cite{BWTopologicalUniqueness}.
	\end{itemize}
	From the above we immediately see that there exists a unique disconnected, stable self-expander $\Gamma_i$ $C^{2,\alpha}$-asymptotic to $\mathcal{C}_i$ (by evolution of entire graph \cite{EHEntireGraph}). We also see that $\mathcal{C}_i \subset \Omega$, where $\Omega$ is the connected component of $\mathbb{R}^{n+1} \setminus \mathcal{C}$ that contains $W$. Denote by $\Sigma_0$ the connected self-expander asymptotic to $\mathcal{C}$. Using a direct method with $\Sigma_0$ as the barrier, similar to Lemma 8.2 of \cite{BWIntegerDegree}, we can find a connected self-expander asymptotic to $\mathcal{C}_i$  in $\Omega'$, where $\Omega'$ is the connected component of $\mathbb{R}^{n+1} \setminus \Sigma_0$ such that the outward unit normal of $\mathcal{C}$ points into $\Omega'$. 
	\par 
	Since there exists a unique disconnected self-expander $\Gamma_i$ asymptotic to $\mathcal{C}_i$, by the partial ordering of self-expanders asymptotic to a fixed cone (Theorem 4.1 of \cite{BWTopologicalUniqueness}), we can pick an innermost connected self-expander $\Sigma_i$ $C^{2,\alpha}$-asymptotic to $\mathcal{C}_i$ (i.e. pick any $\Sigma_i$ such that the only self-expander lying on the inside of $\Sigma_i$ is the disconnected $\Gamma_i$ - note that $\Sigma_i$ might not be unique). We claim that $\Sigma_i$ is unstable. If not, the mountain pass theorem (Corollary 1.2 \cite{BWMountainPass}, this requires $2 \le n \le 6$) and the genericity of $\mathcal{C}_i$ imply the existence of an unstable self-expander $\Sigma'$ lying between $\Sigma_i$ and $\Gamma_i$. $\Sigma'$ must then be connected, but this contradicts the partial ordering. \par 
	Since $\Sigma_i$ is unstable and $\lambda[\mathcal{C}_i] < 2$, we can produce using \cref{nontrivial-flow-lines} an integral Brakke flow $\mathcal{M}^i = \{\mu^i_t\}_{t \in (0,\infty)}$ that moves inwards initially (by expander mean convexity) and satisfies $\lim_{t \to 0} \mu^i_t = \mathcal{H}^n \llcorner \mathcal{C}_i$. Suppose for a contradiction that $\mathcal{M}^i$ is smooth for all $t$, then the flow is expander mean convex (in the classical sense) and moves inwards for all time. Moreover, using an almost identical argument as in Proposition 5.1(3) of \cite{BWTopologicalUniqueness}, the rescaled flow (rescaling as in the proof of \cref{nontrivial-flow-lines}) $\tilde{\mathcal{M}}^i$ converges as $s \to \infty$ to a smooth, stable self-expander asymptotic to $\mathcal{C}_i$, which must lie inside $\Sigma_i$. Since $\Sigma_i$ is an innermost connected self-expander, the stable limit must be $\Gamma_i$ which is disconnected, a contradiction. \par 
	Now let $s^i$ denote the first singular time of the rescaled flows $\tilde{\mathcal{M}}^i$, and time translate $\tilde{\mathcal{M}}^{i}$ by $-s^i$ to obtain rescaled flows $\tilde{\mathcal{M}}^{i,s^i}$ with a singularity at time $0$. By compactness of MCF we obtain a rescaled flow $\tilde{\mathcal{M}}$ such that $\tilde{\mathcal{M}}^{i,s^i} \to \tilde{\mathcal{M}}$ subsequentially. Rescaling back we see that, by upper semicontinuity of the Gaussian density, $\mathcal{M}$ has its first singularity at time $t = 1$. Finally, we claim that $\mathcal{M}$ indeed comes out of the cone $\mathcal{C}$. As $\mathcal{M}$ is smooth on $(0,1)$, it is enough to show that 
	\begin{align}
	\label{link-convergence}
		\lim_{t \to 0} \supp {\mu}_t \cap \mathbb{S}^{n} = \sigma \text{ in } C^{2,\alpha}(\mathbb{S}^n).
	\end{align}
	Since each $\mathcal{M}^i$ attains $\mathcal{C}_i$ as initial data and $\mathcal{C}_i$ is $C^{2,\alpha}$-regular, we have
	\begin{align*}
		\lim_{t \to 0} \supp \mu_t^i \cap \mathbb{S}^{n} = \sigma_i \text{ in } C^{2,\alpha}(\mathbb{S}^{n}).
	\end{align*}
	As $\sigma_i \to \sigma$ in $C^{2,\alpha}(\mathbb{S}^n)$, by a diagonalization argument, we see that \cref{link-convergence} holds. This completes the proof.
\end{proof}
Assuming \cref{yao-conjecture}, the assumptions in \cref{disconnection} are satisfied by cones of the type \cref{good-cone}, given that the parameter $m$ is sufficiently small. In fact, numerical computations do confirm that the cones $x_1^2 = m^2(x_2^2 + x_3^2)$ for $m \le 1$ have entropy less than 2. Moreover, by \cite{AngenentIlmanenChopp}, there exists a connected self-expander for sufficiently small $m$. In these cases, combining the above with \cref{cylindrical-singularity}, we have the much stronger conclusion that any such cone has a potential evolution which disconnects at a cylindrical singularity.

\section{Self expanders with triple junctions}
\label{ode-appendix}
An important question in the study of self-expander is to determine the number of self-expanders coming out of a given cone. A classical result of Ecker--Huisken \cite{EHEntireGraph} shows that there exists a unique self-expander coming out of a graphical cone. In general, however, Angenent--Ilmanen--Chopp \cite{AngenentIlmanenChopp} showed numerically that uniqueness fails for double cones. It is proved rigorously by Helmensdorfer \cite{Helmensdorfer} that there are at least three distinct smooth self-expanders asymptotic to a rotationally symmetric double cone of the form \cref{good-cone} provided the cone angle is sufficiently large (when the cone angle is small, a barrier argument shows that uniqueness indeed holds --- see Lemma 8.1 of \cite{BWIntegerDegree}). \par 
In this section we prove a simple ODE result on the existence of two self-expanders with triple junctions for rotationally symmetric double cones with sufficiently large cone angle. The proof roughly follows the setup of Helmensdorfer \cite{Helmensdorfer}, although we do not need the clearing out lemma for MCF in the following analysis. This provides an example of a singular self-expander, and also illustrates that the cyclicity assumption in \cref{reflection-symmetry} is essential as tameness (\cref{tameness}) clearly fails for self-expanders with triple junction singularities. However, the examples constructed below are in fact still rotationally symmetric. \par 
We consider cones $\mathcal{C}_m$ of the form \cref{good-cone}, where $m$ is the parameter therein. Observe that $\mathcal{C}_m$ has a rotational symmetry across the $x_1$-axis as well as a reflection symmetry across the $\{x_1 = 0\}$ hyperplane. \par 
Assume the expander $\Sigma$ has a triple junction singularity at a point $(0,x_0)$. Imposing rotational symmetry on $\Sigma$ across the $x_1$-axis we may assume $x_0 = (a,0,\ldots,0)$. Note also at a triple junction singularity, the tangent cone is stationary and is therefore the union of three half-lines meeting at an angle of $2\pi/3$. These observations reduce the problem to finding a function $u: \mathbb{R}^+ \to \mathbb{R}$ satisfying the following ODE (written in spherical coordinates):
\begin{align}
	\label{expander-ode}
	\frac{u_{rr}}{1+ u_r^2} - \frac{n-1}{u} + \frac{1}{2}r u_r - \frac{1}{2}u = 0.
\end{align}
with initial data $u(0) = a$ and $u'(0) = \frac{\sqrt{3}}{3}$. The solution is asymptotic to $\mathcal{C}_m$ if 
\begin{align}
	\label{cone-condition}
	\lim_{r \to \infty} \frac{u(r)}{r} = \frac{1}{m}.
\end{align}
Of course, by the usual ODE existence and uniqueness theorem, the solution to the problem \cref{expander-ode} is unique with a given initial data $a$. First we show that any solution to \cref{expander-ode} is asymptotically conical. 
\begin{prop}
	\label{prop-c1}
	For any $a > 0$, there is a (unique) $m = m(a)$ such that \cref{cone-condition} holds.
\end{prop}
\begin{proof}
	First we prove $u > 0$ for all $r > 0$. Suppose for a contradiction that there is $r_0$ such that $u(r_0) < 0$. Since initially $u$ and $u_r$ are both positive, by the mean value theorem there must be a local maximum $r_1 \in (0,r_0)$ with $u(r_1) > 0$, but at such an $r_1$ we can use \cref{expander-ode} to get
	\begin{align*}
		u_{rr}(r_1) = \frac{n-1}{u(r_1)} + \frac{1}{2}u(r_1) > 0,
	\end{align*}
	a contradiction. So $u$ is indeed positive. By the above calculation, this implies that all critical points of $u$ are local minima. \par 
	To continue, let $\alpha(r) = \arctan(u/r)$. Differentiating, we obtain
	\begin{align*}
		\alpha'(r) = \frac{ru_r - u}{u^2 + r^2} \text{ and } \alpha''(r) = \frac{ru_{rr}}{u^2 + r^2} - \frac{(ru_r - u)(2uu_r + 2r)}{(u^2+r^2)^2}.
	\end{align*}
	At a critical point $r_0$ of $\alpha(r)$, we have $r_0u_r(r_0) - u(r_0)= 0$ and \cref{expander-ode} implies that $u_{rr}(r_0) = \frac{n-1}{u(r_0)} > 0$. Hence $\alpha''(r_0) = \frac{ru_{rr}}{u^2 + r^2} > 0$. Therefore all critical points of $\alpha(r)$ are local minima as well. Since $\alpha(r)$ is bounded and only has local minima, monotone convergence theorem shows that $\lim_{r \to \infty} \alpha(r)$ exists. \par 
	\par
	It remains to show that the $\lim_{r \to \infty} \alpha(r) \in (0,\frac{\pi}{2})$. If the limit is 0, then $m = 0$ and the MCF of $\Sigma$ is contained in the level set flow of the $x_1$-axis, which disappears immediately. This is impossible. If the limit is $\frac{\pi}{2}$, the MCF of $\Sigma$ is contained in the level set flow of the hyperplane $\{x_1 = 0\}$, which is static. This is again impossible by, say, the initial condition $u_r(0) = \frac{\sqrt{3}}{3}$.
\end{proof}
Knowing the above, our problem becomes essentially a shooting problem: Given $m$, we wish to find the appropriate initial condition $a$ so that the solution to \cref{expander-ode} satisfies \cref{cone-condition}. Let $u^a(r)$ be the solution to \cref{expander-ode} with initial condition $u^a(0) = a$. Consider the asymptotic cone angle parameter $m$ as a function of $a$: 
\begin{align*}
	m(a) = \lim_{r \to \infty} \frac{u^a(r)}{r} = \lim_{r \to \infty} u_r^a(r) \in (0,\infty).
\end{align*}
\begin{prop}
	\label{continuity}
	$m(a)$ is a continuous function on $(0,\infty)$.
\end{prop}
\begin{proof}
	First we record that
	\begin{align*}
		u_{rr}(0) = \frac{4}{3} \left(\frac{n-1}{a} + \frac{1}{2}a\right) > 0.
	\end{align*}
	From the proof of \cref{prop-c1}, we see that a critical point of $u$ must be a local minimum, but since $u$ is initially increasing and smooth, there cannot be any critical point at all. So $u$ is strictly increasing and we deduce that $u_r > 0$ for all $r > 0$. On the other hand, l'H\^{o}pital's rule on \cref{cone-condition} yields $\lim_{r \to \infty} u_r = m$. Going back to \cref{expander-ode} and taking the limit as $r \to \infty$ yield also $\lim_{r \to \infty} u_{rr} = 0$. \par 
	Fix an $a \in (0,\infty)$. Clearly $u^a \ge a$ and $u_r^{a}$ is bounded above. Therefore we may fix a constant $N$ with $\frac{1}{N} < a$ such that $\abs{a' - a} < \frac{1}{N}$ implies that 
	\begin{align}
		\label{uniform-bound}
		\frac{1}{u^{a'}} \le c 
	\end{align}
	for some constant $c$ depending on $a$ and $N$. We compute 
	\begin{align}
		\label{estimate-2}
		\left(\frac{u}{r}\right)_r = \frac{1}{r}\left(u_r - \frac{u}{r}\right) = \frac{2}{r^2}\left(\frac{n-1}{u} - \frac{u_{rr}}{1 + (u_r)^2}\right).
	\end{align}
	For $\abs{a - a'} < \frac{1}{N}$ there are two cases. If $u_{rr}^{a'}$ is never zero, then it is always positive. This immediately gives the bound:
	\begin{align}
		\label{upper-bound}
		\left(\frac{u^{a'}}{r}\right)_r \le \frac{c}{r^2}.
	\end{align}
	If $u_{rr}^{a'}(r_0) = 0$ at some point $r_0$, differentiating \cref{expander-ode} once we get 
	\begin{align}
		\label{urrr}
		\frac{1}{1+u_r^2}\left(u_{rrr} - \frac{2u_r(u_{rr})^2}{1+u_r^2}\right) + \frac{n-1}{u^2} u_r + \frac{1}{2}ru_{rr} = 0.
	\end{align}
	From this we immediately see that when $u_{rr}^{a'}(r_0) = 0$, 
	\begin{align*}
		u_{rrr}^{a}(r_0) = -(1+u_{r}^2)\frac{n-1}{u^2} < 0,
	\end{align*}
	so every critical point of $u_r^{a'}$ is a local maximum, for which there can be at most one of them. Therefore $r_0$ is the only zero of $u_{rr}^{a'}$. In this case, we see from \cref{urrr} that, at any negative local minimum of $u_{rr}^{a'}$, we have
	\begin{align*}
		\frac{2u_r^{a'}}{(1 + (u_r^{a'})^2)^2} (u_{rr}^{a'})^2 \le  \frac{(n-1)u_r^{a'}}{(u^{a'})^2} \implies \frac{ (u_{rr}^{a'})^2}{(1 + (u_r^{a'})^2)^2} \le \frac{n-1}{2(u^{a'})^2} \le c,
	\end{align*}
	where we used \cref{uniform-bound}. Going back to \cref{estimate-2}, this yields the same type of uniform upper bound as \cref{upper-bound} (up to increase $c$). Therefore we conclude \cref{upper-bound} holds for all $\abs{a - a'} < \frac{1}{N}$. \par 
	Integrating \cref{upper-bound} from $r$ to $\infty$, we obtain the estimate: 
	\begin{align*}
		m(a') - \frac{u^{a'}(r)}{r} < \frac{c}{r}, \; \; \abs{a' - a} < \frac{1}{N}.
	\end{align*}
	Now given $\varepsilon > 0$ we can pick $r_0 > 0$ such that $c/r_0 < \varepsilon/3$ and $\delta$ so small that $\abs{a' - a} < \delta$ implies that $\abs{u^{a'} - u^a} < \varepsilon/3$ on $(0,r_0]$ (this follows from continuous dependence on initial data as we are now in a compact set). Using the triangle inequality we get that
	\begin{align*}
		\abs{m(a) - m(a')} \le \abs{m(a) - \frac{u^a(r_0)}{r_0}} + \abs{\frac{u^a(r_0)}{r_0} - \frac{u^{a'}(r_0)}{r_0}} + \abs{m(a') - \frac{u^{a'}(r_0)}{r_0}} < \varepsilon.
	\end{align*}
	This finishes the proof of continuity.
\end{proof}
We are now in the position to prove the existence theorem.
\begin{thm}
\label{existence-singular-expander}
	There is an $M_0 > 0$ such that for all $M > M_0$, there exists at least two distinct values $a_1, a_2 \in (0,\infty)$ depending on $M$ such that $m(a_1) = m(a_2) = M$.
\end{thm}
\begin{proof}
	We will show that $m(a) \to \infty$ both as $a \to 0$ and as $a \to \infty$. In view of \cref{continuity} this will prove the theorem.  \par 
	Let us show that $m(a) \to \infty$ as $a \to \infty$. First of all, we show that $u_r^a$ cannot be uniformly bounded above. Indeed if $u_r^a$ is uniformly bounded above by some constant $C$, then $u - ru_r \ge a - C$ for $r \in [0,1]$ and so $u_{rr} \ge \frac{1}{2}(a-C)$ on $[0,1]$ by \cref{expander-ode}. But then 
	\begin{align*}
		u^a_r(1) = \frac{\sqrt{3}}{3} + \int_0^1 u^a_{rr}(r)dr \ge \frac{\sqrt{3}}{3} + \frac{1}{2}(a - C) \to \infty 
	\end{align*}
	as $a \to \infty$, a contradiction. \par 
	Recall from the proof of \cref{continuity} that $u_{rr}^{a}$ can have at most one zero. If there is a sequence $a_i \to \infty$ such that $u_{rr}^{a_i}$ has one zero $r_i$, then \cref{estimate-2} immediately implies that 
	\begin{align*}
		\left(\frac{u^{a_i}}{r}\right)_r > 0, r > r_i.
	\end{align*}
	Integrating the above from $r_i$ to $\infty$ we get 
	\begin{align}
		\label{estimate-1}
		m(a_i) > \frac{u^{a_i}(r_i)}{r_i}.
	\end{align}
	On the other hand, at a zero of $u_{rr}$, \cref{expander-ode} gives that
	\begin{align}
		\label{expander-ode-2}
		\frac{u^{a_i}(r_i)}{r_i} = u^{a_i}_r(r_i)- \frac{2(n-1)}{u^{a_i}(r_i)r_i} 
	\end{align}
	Observe that $u_r^{a_i}(r_i) = \sup_{r} u_r^{a_i}(r)$ because $r_i$ is a local maximum of $u^{a_i}_r$ and $u^{a_i}_{rr}(r) < 0$ for all $r > r_i$. Since $u_r^{a_i}$ is not uniformly bounded, we have that \begin{align*}
		\lim_{i \to \infty} \frac{u^{a_i}(r_i)}{r_i} = \lim_{i \to \infty} u^{a_i}_r(r_i)- \frac{2(n-1)}{u^{a_i}(r_i)r_i} = \infty, 
	\end{align*}
	where we also used $u^{a_i}(r_i) > a_i$. Recalling \cref{estimate-1}, we see that $\lim_{i \to \infty} m(a_i) = \infty$. \par
	Otherwise there is $a_0 > 0$ such that $u_{rr}^a$ has no zero for all $a > a_0$. This means that $u_r^a$ is strictly increasing for all $a > a_0$. Since $u_r^a$ is not uniformly bounded we can find a sequence $a_i \to \infty$ and $\{r_i\} \subset [0,\infty)$ such that $u_r^{a_i}(r_i) > i$. Of course monotonicity implies that $m(a_i) \ge u_r^{a_i}(r_i) > i$. This shows that $m(a) \to \infty$ as $a \to \infty$. \par 
	Next we will show that $m(a) \to \infty$ as $a \to 0$ as well. Again we will first argue that $u_r^a$ cannot be uniformly bounded near $0$. Suppose for a contradiction that there is $a_0 > 0$ such that $\abs{u_r^a} \le C$ for all $0 < a < a_0$, then since $u_r^a > 0$ we get 
	\begin{align}
		\label{upper-bound-2}
		u^a(r) = a + \int_0^r u_r^a(t) dt \le a + Cr, \;\; a < a_0.
	\end{align}
	On the other hand, \cref{expander-ode} gives that 
	\begin{align*}
		u_{rr}^a \ge \frac{n-1}{u^a} + \frac{1}{2}(u^a - r u_r^a) \ge \frac{n-1}{2u^a} + \sqrt{n-1} - Cr, \;\; a < a_0.
	\end{align*}
	In particular for sufficiently small $a$ we can ensure $\sqrt{n-1} \ge Cr$ and so that $u_{rr}^a(r) \ge \frac{n-1}{2u^a(r)}$ for $r < \sqrt{a}$. Hence, using \cref{upper-bound-2}, we may estimate
	\begin{align*}
		u_{r}^a(\sqrt{a}) &\ge \frac{\sqrt{3}}{3} + \int_0^{\sqrt{a}} \frac{n-1}{2u^a(r)} dr \\
		&\ge \frac{\sqrt{3}}{3} + \frac{n-1}{2}\int_0^{\sqrt{a}} \frac{1}{a + Cr} dr = \frac{\sqrt{3}}{3} + \frac{n-1}{2C} \log(1 + Ca^{-\frac{1}{2}}) \to \infty
	\end{align*}
	as $a \to 0$, a contradiction. This shows that $u_r^a$ is not uniformly bounded near $0$. \par 
	Suppose there is a sequence $a_i \to 0$ such that $u_{rr}^{a_i}$ has one zero $r_i$. In view of \cref{expander-ode-2}, if $r_i$ is bounded away from 0, then taking $i \to \infty$ will give  
	\begin{align*}
		\lim_{i \to \infty} \frac{u^{a_i}(r_i)}{r_i} = \lim_{i \to \infty} u^{a_i}_r(r_i)- \frac{2(n-1)}{u^{a_i}(r_i)r_i} = \infty,
	\end{align*}
	where we used the fact that $u^{a_i}(r_i) \ge a_i + \frac{\sqrt{3}}{3}r_i$ and that $u^{a}_r$ is not uniformly bounded near 0. By \cref{estimate-1}, we can conclude $m(a_i) \to \infty$ as $i \to \infty$ as before. So we henceforth assume that $r_i \to 0$ and assume for a contradiction that there is a constant $C$ such that $r^{-1}u^{a_i}(r) \le C$ uniformly for $r \ge r_i$. Since $u_r$ is decreasing on $[r_i,\infty)$ we get
	\begin{align*}
		u^{a_i}(2r_i) \ge \int_{r_i}^{2r_i} u^{a_i}_r(t) dt \ge r_i u^{a_i}_r(2r_i) \implies u_r^{a_i}(2r_i) \le \frac{u^{a_i}(2r_i)}{r_i} \le 2C.
	\end{align*}
	Since $u^{a_i}_{rr}(2r_i) < 0$, \cref{expander-ode} gives that
	\begin{align*}
		0 < \frac{u^{a_i}(2r_i)}{2r_i} \le \frac{-2(n-1)}{2r_iu^{a_i}(2r_i)} + u^{a_i}_r(2r_i).
	\end{align*}
	Rearranging, we deduce that $r_iu^{a_i}(2r_i) \ge (n-1)u_r^{a_i}(2r_i)^{-1} \ge \frac{1}{2}(n-1)C^{-1}$, and so 
	\begin{align*}
		2Cr_i^2 \ge r_iu^{a_i}(2r_i) \ge \frac{1}{2}(n-1)C^{-1} \implies r_i^2 \ge \frac{1}{4}(n-1)C^{-2}
	\end{align*}
	a contradiction as $r_i \to 0$. Hence $r^{-1}u^{a_i}(r)$ is not uniformly bounded on $(r_i,\infty)$, so for each $i$ we may find $r'_{i} \in (r_i,\infty)$ such that $(r_i')^{-1}u^{a_i}(r'_i) > i$. \cref{estimate-1} then implies $m(a_i) \to \infty$ as $i \to \infty$. \par 
	Otherwise $u_{rr}^{a}$ has no zero for sufficiently small $a$ and monotonicity as before implies that $m(a) \to \infty$ as $a \to 0$. This completes the proof.
\end{proof}

\section{Further remarks}
\label{further-remarks}
We conclude our article with some open questions and conjectures, some of which we have already alluded to before.  \par 
The most natural question to ask is what happens if $\mathcal{C}$ is rotationally symmetric and the link $\mathcal{L}(\mathcal{C})$ has three (or more) connected components. It is not expected that self-expanders asymptotic to $\mathcal{C}$ will be rotationally symmetric. This should be compared to the case of minimal surfaces - the Costa surface, which is not rotationally symmetric, has two catenoidal ends and one planar end. It is therefore natural to expect that something similar happens if the cone is given by the union of a rotationally symmetric double cone and a hyperplane. We suspect that gluing method by desingularizing the connected expander (asymptotic to the cone) with the hyperplane will produce a counterexample. \par 
Another important problem is to determine the entropy of rotationally symmetric double cones. In fact, we conjecture that the assumption $\lambda[\mathcal{C}] < 2$ in \cref{main-theorem} is redundant (where as the same assumption in \cref{reflection-symmetry} is essential). More precisely we conjecture:
\begin{conj}
\label{yao-conjecture}
	Let $\mathcal{C}$ be a rotationally symmetric double cone of the form:
	\begin{align*}
		x_1^2 =  \begin{cases} m_1(x_2^2 + x_3^2 + \cdots + x_n^2), & x_1 \ge 0 \\  m_2(x_2^2 + x_3^2 + \cdots + x_n^2), & x_1 < 0\end{cases},
	\end{align*}
	where $m_1, m_2 > 0$. Then $\lambda[\mathcal{C}] < 2$. 
\end{conj}
When $m_1 = m_2 = m$, observe that when $m \to 0$, $\mathcal{C}$ converges to a multiplicity 2 plane which has entropy 2, and that when $m \to \infty$, $\mathcal{C}$ converges after suitable translations to a cylinder $\mathbb{R} \times \mathbb{S}^{n-1}$ which has entropy strictly less than 2. This also explains why the two connected components of $\mathcal{L}(\mathcal{C})$ need to be in different half-spaces, for otherwise when $m$ is small the double cone could be close to a multiplicity two cylinder which has entropy strictly larger than 2. \par 
It is not so hard to calculate the Gaussian area of $\mathcal{C}$ at the origin, but unlike self-shrinkers, the maximum in our case needs not to happen at the origin (the argument for self-shrinkers can be found in eg. Section 7 of \cite{CMGeneric}). In fact when $m$ is large, the entropy is achieved far away from the origin (in a region where the cone looks more like a cylinder). Although we have strong numerical evidence that the conjecture is true, the analysis of the Gaussian area functional on the cone centered away from the origin is complicated to handle. \par 
Less clear is the entropy of cones with $O(p + 1) \times O(n - p + 1)$ symmetry.  Ilmanen and White \cite{IlmanenWhite} give an exact formula for the Gaussian density of $\mathcal{C}_{n,p}$ at the origin:
\begin{align*}
	\Theta_{\mathcal{C}_{n,p}}(0) = \frac{\sigma_{p} \sigma_{n-p}}{\sigma_{n}} \left(\frac{p}{n}\right)^{p/2}\left(\frac{n-p}{n}\right)^{(n-p)/2},
\end{align*}
where $\sigma_p$ is the volume of the unit sphere in $\mathbb{R}^{p+1}$. It can be checked this is less than 2 for all $n,p$ (the proof is by rather tedious computation so we omit it here, but one can easily verify by numerics as well). Since $\mathcal{C}_{n,p}$ are minimal, they are also self-shrinkers. It follows from a theorem of Colding--Minicozzi \cite{CMGeneric} that the entropy of $\mathcal{C}_{n,p}$ is achieved at the origin. Hence, in fact, $\lambda[\mathcal{C}_{n,p}] = \Theta_{\mathcal{C}_{n,p}}(0) < 2$. For example, the cone $\mathcal{C}_{2,1} \subset \mathbb{R}^4$ has entropy $\frac{3}{2}$. It is therefore reasonable to make the following conjecture, partly due to Solomon:
\begin{conj}[cf. Section 4 of \cite{IlmanenWhite}]
	Any double cone with $O(p+1) \times O(n - p + 1)$ symmetry has entropy at least that of $\mathcal{C}_{n,p}$ and at most 2.
\end{conj}
Finally we discuss the rotational symmetry of singular self-expanders. As we remarked before, the Hopf lemma \cref{hopf-lemma} cannot deal with triple junction singularities which are a possibility as we have seen in \cref{ode-appendix}.
\begin{ques}
	What can we say about the symmetry of a singular self-expander coming out of a rotationally symmetric double cone? \par 
	In particular, if $\Sigma$ is a self-expander asymptotic to $\mathcal{C}_m$ (notation as in \cref{ode-appendix}) that is smooth away from $\{x_1 = 0\}$ and only has triple junction singularities on $\{x_1 = 0\}$, is $\Sigma$ rotationally symmetric?
\end{ques}
It is possible that one can use a parabolic variant of the argument of Bernstein--Maggi \cite{BernsteinMaggi} for singular Plateau surfaces (in particular Lemma 2.4 therein) to prove the desired rotational symmetry.

\appendix

\section{Construction of the Smooth Flow}
\label{smooth-construction-appendix}
Here we give a more detailed overview of the construction of smooth Morse flow lines from \cref{construction-of-smooth-flow}. The construction is an adaptation of the work of Bernstein and Wang in a series of papers, including \cite{BWMountainPass}, \cite{BWRelativeEntropy}, \cite{BWTopologicalUniqueness} and \cite{BWTopology}. \par 
Throughout the section, we assume $\mathcal{C} \subset \mathbb{R}^{n+1}$ is a smooth cone. The first proposition was originally proven for self-shrinkers in \cite{BWTopology} and we give a modified proof for self-expanders.
\begin{prop}[cf. Lemma 4.3 of \cite{BWTopology}]
	\label{distance-estimate}
	Suppose $\Sigma$ is an hypersurface $C^{2,\alpha}$-asymptotic to $\mathcal{C}$, and there is $N > 0$ such that $\Sigma \setminus B_{NR}(0) \subset \mathcal{T}_{R^{-1}}(\mathcal{C})$ for $R > 1$. If $\{\Sigma_t\}_{t \in [1,T]}$ is an integral Brakke flow starting from $\Sigma$ in $\mathbb{R}^{n+1}$, then there is a constant $N' > 0$ such that 
	\begin{align*}
		\Sigma_t \setminus B_{N'R \sqrt{t}}(0) \subset \mathcal{T}_{R^{-1}\sqrt{t}}(\mathcal{C})
	\end{align*}
	for $R > 1$.
\end{prop}

\begin{proof}
	For any $x \in \mathbb{R}^{n+1} \setminus (B_{NR}(0) \cup \mathcal{T}_{R^{-1}}(\mathcal{C}))$ let 
	\begin{align*}
		\rho(x) = \inf\{ \rho' \ge 0 \mid B_{\rho'}(x) \cap (B_{NR}(0) \cup \mathcal{T}_{R^{-1}}(\mathcal{C})) \ne \emptyset\}
	\end{align*}
	and let 
	\begin{align*}
		\rho_t(x) = \begin{cases} 
			\sqrt{\rho(x)^2 - 2n(t-1)} & \rho(x)^2 \ge 2n(t-1) \\ 0 & \rho(x)^2 < 2n(t-1)
		\end{cases}
	\end{align*}
	be the corresponding MCF starting from $B_{\rho}(x)$. Let 
	\begin{align*}
		U_t = \bigcup_{\rho_t(x) > 0} B_{\rho_t(x)}(x).
	\end{align*}
	be the time $t$ slice of the above MCF. Since initially $\Sigma \cap U_1 = \emptyset$ maximum principle implies that $U_t \cap \Sigma_t = \emptyset$ for all $t \in [1,T)$. This proves that 
	\begin{align}
		\Sigma_t \setminus B_{NR+ \sqrt{2n(t-1)}}(0) \subset \mathcal{T}_{R^{-1} + \sqrt{2n(t-1)}}(\mathcal{C})
		\label{eq31}
	\end{align}
	since $\mathbb{R}^{n+1} \setminus (B_{NR+ \sqrt{2n(t-1)}}(0) \cup \mathcal{T}_{R^{-1} + \sqrt{2n(t-1)}}(\mathcal{C})) \subset U_t$.\par 
	Next we consider the map $\Phi: \mathcal{C} \setminus \{0\} \times \mathbb{R} \to \mathbb{R}^{n+1}$ given by $\Phi(x,\lambda) = x + \lambda \nu_\mathcal{C}$. Choose $\lambda_0 < 1/2$ depending on $\mathcal{C}$ small enough so that $\Phi|_{\mathcal{L}(\mathcal{C}) \times (-2\lambda_0,2\lambda_0)}$ is a diffeomorphism onto its image. It follows, since $\mathcal{C}$ is a cone, that $\Phi$ is a diffeomorphism on the set 
	$\{(x,\lambda) \mid \abs{\lambda} < 2\lambda_0 \abs{x}\}$. \par 
	Consider, for some $N'$ to be chosen later,
	\begin{align*}
		V_t = \mathbb{R}^{n+1} \setminus (U_t \cup  B_{N'R\sqrt{t}}(0)).
	\end{align*}
	We first claim that $N'$ can be chosen so that $y \in V_t$ can be written as $x + \lambda'\abs{x} \nu_\mathcal{C}(x)$ for some $\abs{\lambda'} < \lambda_0$. Indeed, since $y \not \in U_t$ we have that $\dist(y, \mathcal{C}) \le R^{-1} + \sqrt{2n(t-1)}$ provided we choose $N'$ large enough so that 
	\begin{align*}
		N'R\sqrt{t} > NR + \sqrt{2n(t-1)}.
	\end{align*}
	Let $x$ be the nearest point projection of $y$ onto $\mathcal{C}$, then using $R > 1$ we have that
	\begin{align*}
		\frac{\abs{y-x}}{\abs{x}} &\le \frac{1 + \sqrt{2nt}}{N'\sqrt{t} - 1 - \sqrt{2nt}} = \frac{(\sqrt{t})^{-1} + \sqrt{2n}}{N' - (\sqrt{t})^{-1} - \sqrt{2n}} < \frac{1+\sqrt{2n}}{N' - 1 - \sqrt{2n}} < \lambda_0
	\end{align*}
	provided we choose $N'$ large depending on $n$ only. Hence the claim holds. \par 
	Let $y \in \Sigma_t \setminus B_{N'R\sqrt{t}}(0) \subset V_t$. We claim that, up to further increasing $N'$, $\dist(y, \mathcal{C}) < R^{-1}\sqrt{t}$ and this will finish the proof. To this end consider $y_0 = x + \lambda_0  \abs{x} \nu_\mathcal{C}$. We have that
	\begin{align*}
		\abs{y_0} = \abs{x + \lambda_0 \abs{x}\nu_\mathcal{C}} = \sqrt{1 + \lambda_0^2}\abs{x} > NR + \lambda_0\abs{x} - \frac{1}{3}R^{-1},
	\end{align*}
	if $N'$ is chosen large enough so that
	\begin{align*}
		\abs{x} \ge (N'R - \sqrt{2n})\sqrt{t} - R^{-1} > 4NR - R^{-1},
	\end{align*}
	where we used the fact that $\sqrt{1 + \lambda_0^2} - \lambda_0 > \frac{1}{3}$.
	This implies that $B_{\lambda_0\abs{x} - R^{-1}}(y_0) \cap B_{NR}(0) = \emptyset$. Since $\dist(y_0,\mathcal{C}) = \lambda_0 \abs{x}$ we conclude that
	\begin{align*}
		\rho(y_0) = \lambda_0\abs{x} - R^{-1}.
	\end{align*}
	Increasing $N'$ if necessary we can also ensure that
	\begin{align*}
		\lambda_0\abs{x} - R^{-1} > \sqrt{2n(t-1)}.
	\end{align*}
	Since $y \not\in B_{\rho_{t}(y_0)}(y_0) \subset U_t$ we can estimate
	\begin{align*}
		\dist(y,\mathcal{C}) &\le \dist(y_0,\mathcal{C}) - \rho_t(y_0) = 
		\lambda_0 \abs{x} - \sqrt{(\lambda_0 \abs{x} - R^{-1})^2 - 2n(t-1)}.
	\end{align*}
	To finish the proof we compute, for large $N'$, 
	\begin{align*}
		&\phantom{{}={}}(\lambda_0 \abs{x} - R^{-1}\sqrt{t})^2 - ((\lambda_0 \abs{x} - R^{-1})^2 - 2n(t-1)) \\
		&= 2R^{-1}\lambda_0 \abs{x} (1 - \sqrt{t}) + R^{-2}(t - 1) + 2n(t-1) \\
		&= R^{-1}(\sqrt{t} - 1)(-2\lambda_0 \abs{x} + (R^{-1} + 2nR)(\sqrt{t}+1)) \le 0.
	\end{align*}
	which is equivalent to 
	\begin{align*}
		\lambda_0 \abs{x} - R^{-1}\sqrt{t} \le \sqrt{(\lambda_0 \abs{x} - R^{-1})^2 - 2n(t-1)} \implies \dist(y,\mathcal{C}) \le R^{-1}\sqrt{t} 
	\end{align*}
	provided $N'$ is chosen large enough so that
	\begin{align*}
		\lambda_0 \abs{x} \ge (N'R - \sqrt{2n})\sqrt{t} - R^{-1} \ge (R^{-1} + 2nR)\sqrt{t}. &\qedhere
	\end{align*}
\end{proof}

The next proposition shows the desired regularity for an asymptotically conical MCF. Since the proof is used repeatedly in our presentation, we have also included a proof for the sake of completeness. In the proof we will use the notation $[f]_{\alpha;\Omega}$ to denote the $\alpha$-H\"{o}lder seminorm of $f$ on $\Omega$ (note $\Omega$ could be a subset of $\mathbb{R}^{n+1}$ or a time interval), i.e.
\begin{align*}
	[f]_{\alpha;\Omega} = \sup_{x, y \in \Omega, x \ne y} \frac{\abs{f(x) - f(y)}}{\abs{x - y}^\alpha}.
\end{align*}
\begin{prop}[Lemma 5.3(2) of \cite{BWTopologicalUniqueness}]
	\label{interior-estimates}
	Suppose $\Sigma$ is an hypersurface $C^{2,\alpha}$-asymptotic to $\mathcal{C}$ and let $\{\Sigma_t\}_{t \in [1,T)}$ be a MCF starting from $\Sigma$. If there is $N > 0$ such that $\Sigma_1 \setminus B_{NR}(0) \subset \mathcal{T}_{R^{-1}}(\mathcal{C})$ for all $R > 1$, then there is $N' > 0$ such that $\Sigma_t \setminus B_{N'\sqrt{t}}(0)$ can be written as a (smooth) normal graph over $\mathcal{C} \setminus B_{N'\sqrt{t}}(0)$. In particular we have the uniform curvature bound 
	\begin{align*}
		\sup_{t \in [1,T)} \sup_{\Sigma_t \setminus B_{N'\sqrt{t}}} \abs{A_{\Sigma_t}} < \infty.
	\end{align*}
\end{prop}
\begin{proof}
	Fix $t \in [1,T)$ and let $\delta > 0$. Since $\Sigma$ is asymptotically conical, by \cref{expander-curvature-estimates} there is $N_1$ and $\varepsilon > 0$ depending on $\delta$ such that for any $x_0 \in \mathcal{C} \setminus B_{N_1}(0)$, $\Sigma_1 \cap C_{\eta}(x_0)$ can be written as a graph $\tilde{f}_{x_0}(x)$ over some neighborhood of $x_0$ in $T_x\mathcal{C}$ containing $B^n_{\eta}(x_0)$. Here $\eta = \varepsilon\abs{x}$. Moreover up to increasing $N_1$ we can ensure \cref{distance-estimate} holds, and that $\tilde{f}_{x_0}$ satisfies the estimates
	\begin{align}
		\label{appendix-eq-1}
		\sum_{i=0}^2 \eta^{-1+i} \sup_{B^n_\eta(x_0)} |\nabla^i \tilde{f}_{x_0}| + r^{1 + \alpha} [\nabla^2 \tilde{f}_{x_0}]_{\alpha} < \delta
	\end{align}
	By pseudolocality of MCF (Theorem 1.5 in \cite{IlmanenNevesSchulze}), given $\varepsilon > 0$, there is $N_1 > 0$ such that for every $x_0 \in \Sigma_1 \setminus B_{N_1}(0)$ and $s \in [1,t]$, $\Sigma_s \cap C_{\eta/2}(x_0)$ can be written as a normal graph over $\Sigma_1 \cap B^n_{\eta/2}(x_0)$. Combining the above two facts, we see that for sufficiently small $\delta$ and $\varepsilon$, $\Sigma_s \cap C_{\eta/2}(x_0)$ can be written as the graph of a function $f_{x_0}(s,x)$ over some neighborhood of  $T_{x_0}\mathcal{C}$ for all $s \in [1,t]$ and $x_0 \in \mathcal{C} \setminus B_{N_1}(0)$. Moreover, $f_{x_0}$ satisfies the pointwise estimates 
	\begin{align*}
		(\eta/2)^{-1} \sup_{B^n_{\eta/2}(x_0)} \abs{f_{x_0}(s,\cdot)} + \sup_{B^n_{\eta/2}(x_0)} \abs{\nabla_x f_{x_0}(s,\cdot)} < 1.    
	\end{align*}
	For the rest of the proof we fix an $x_0$ and put $f = f_{x_0}$. Since $\{\Sigma_s\}_{s \in [1,t]}$ is a graphical MCF near $x_0$, $f$ satisfies the evolution equation:
	\begin{align*}
		\frac{\partial f}{\partial s} = \sqrt{1 + \abs{\nabla_{x} f}^2} \Div\left(\frac{\nabla_{x}f}{\sqrt{1 + \abs{\nabla_{x} f}^2}}\right).
	\end{align*}
	This is a quasilinear parabolic equation, so we may use H\"{o}lder estimates (eg. Theorem 1.1 in Chapter 4 of \cite{LSU}) to get that 
	\begin{align*}
		\sup_{s \in [1,t]} [\nabla_x f(s,\cdot)]_{\alpha;B_{\eta/4}(x_0)} + \sup_{B^n_{\eta/4}(x_0)}[\nabla_x f(\cdot,x)]_{\alpha/2;[1,t]} \le C(\eta/4)^{-\alpha}.
	\end{align*}
	for any $\alpha \in (0,1)$. Standard Schauder estimates (see eg. Chapter 5 of \cite{Lieberman} or Theorem 5.1 in Chapter 4 of \cite{LSU}) yield higher order estimates of the form
	\begin{align*}
		\sum_{i=0}^2 (\eta/8)^{i-1}\sup_{B^n_{\eta/8}(x_0)}\abs{\nabla^i f(s,\cdot)} + (\eta/8)^{1 + \alpha} [\nabla^2 f(s,\cdot)]_{\alpha;B_{\eta/8}(x_0)} \le C
	\end{align*}
	for $s \in [1,t]$ and 
	\begin{align*}
		\sup_{x \in B^n_{\eta/8}(x_0)} [\nabla_x f_{x_0}(s,x)]_{\frac{1}{2};[1,t]} \le C (\eta/8)^{-1}.
	\end{align*}
	From the above we may estimate
	\begin{align*}
		\abs{f_{x_0}(s,x) - f_{x_0}(1,x_0)} \le C(s-1)(\eta/8)^{-1} + \delta \abs{x-x_0} +  C(\eta/8)^{-1}\abs{x-x_0}^2
	\end{align*}
	where we also used the evolution equation and the fact that $\abs{\nabla_x f(1,x_0)} < \delta$ from \cref{appendix-eq-1}. These implies that, for $\rho < 1/8$, a fixed $s \in [1,t]$ and $x_0 \in \mathcal{C} \setminus B_{\tilde{N}\sqrt{s}}(0)$,
	\begin{align*}
		(\rho \eta)^{-1}\sup_{x \in B_{\rho \eta}^n(x_0)} \abs{f(s,x)} &\le (\rho\eta)^{-1}N_1\abs{x_0}^{-1} +  C(s-1)\rho^{-1}\eta^{-2} + \delta + C\rho \\
		&\le  \frac{(\rho \varepsilon)^{-1} N_1}{\tilde{N}^2 s} + \frac{C(s-1)\rho^{-1}}{\tilde{N}^2 s} + \delta + C\rho
	\end{align*}
	where we used that $\abs{f(1,x_0)} < N_1 \abs{x_0}^{-1}$ by \cref{distance-estimate}. The right hand side of the above equation can be made arbitrarily small provided we choose $\delta$ and $\rho$ small enough and $\tilde{N}$ large enough. Similarly we can estimate the derivative 
	\begin{align*}
		\abs{\nabla_{x_0} f_{x_0}(s,x) - f_{x_0}(1,x_0)} \le C(\eta/8)^{-1} \abs{x - x_0} + C(\eta/8)^{-1} \sqrt{s - 1}
	\end{align*}
	and 
	\begin{align*}
		\sup_{x \in B^n_{\rho\eta}(x_0)}\abs{\nabla_x f(s,x)} \le \delta + C\rho + C\frac{\sqrt{s-1}}{\tilde{N} \sqrt{s}}
	\end{align*}
	which can be made arbitrarily small as well. These two decay estimates together with Schauder estimates above with $\eta/8$ replaced by $\rho \eta$ give
	\begin{align*}
		\sum_{i=0}^2 (\eta/8)^{i-1}\sup_{B^n_{\eta/8}(x_0)}\abs{\nabla^i f(s,\cdot)} + (\eta/8)^{1 + \alpha} [\nabla^2 f(s,\cdot)]_{\alpha;B_{\eta/8}(x_0)} \le \frac{1}{2} + C(\rho + \rho^{1 + \alpha}) 
	\end{align*}
	which can be made to be less than 1 provided $\rho$ is chosen small enough. This proves that $\Sigma_t \cap C_{\rho\eta}(x_0)$ is a graph over (a neighborhood of) $T_{x_0}\mathcal{C}$ for $x_0 \in \mathcal{C} \setminus B_{\tilde{N}\sqrt{t}}(0)$ with derivative bounds up to the second order. The curvature bounds follow easily, and higher order bounds follow similarly using Schauder estimates.
\end{proof}

Given a MCF $\{\Sigma_t\}_{t \in I}$, the \textit{expander mean curvature} of $\Sigma_t$ is  
\begin{align*}
	E_{\Sigma_t}(x) = 2tH_{\Sigma_t} + \inner{x,\nu_{\Sigma_t}}.
\end{align*}
We say $\{\Sigma_t\}$ is \textit{expander mean convex} if the $E_{\Sigma_t}(x) > 0$ along the flow. For a fixed time $t$ and a hypersurface $\Sigma$, the relative expander mean curvature of $\Sigma$ is 
\begin{align*}
	E_\Sigma^t(x) = 2tH_\Sigma + x^\perp.
\end{align*}
For $\beta > 0$ define the auxiliary function $g_\beta: \mathbb{R}^+ \to \mathbb{R}^+$ by
\begin{align*}
	g_\beta(s) = s^{-\beta} e^{-\beta s}.
\end{align*}
We now prove the main existence theorem for expander mean convex hypersurfaces without entropy bound.
\begin{thm}[Existence Theorem, cf. Proposition 5.1 of \cite{BWTopologicalUniqueness}]
	\label{existence-theorem}
	Let $\Sigma$ be a hypersurface $C^{2,\alpha}$-asymptotic to $\mathcal{C}$ with no closed components. Suppose that there is $N$ such that $\Sigma \setminus B_{NR}(0) \subset \mathcal{T}_{R^{-1}}(\mathcal{C})$ and that there is $c,\beta > 0$ such that 
	\begin{align*}
		E_\Sigma(x) \ge cg_\beta (1 + \abs{x}^2) > 0, \;\; x\in \Sigma.
	\end{align*}
	Then there exists a unique MCF $\{\Sigma_t\}_{t \in [1,T)}$ with $\Sigma_1 = \Sigma$, where $T$ is the first singular time (possibly $\infty$). Moreover the MCF satisfies
	\begin{enumerate}
		\item $\Sigma_t$ is $C^{2,\alpha}$-asymptotic to $\mathcal{C}$ for all $t \in [1,T)$.
		\item $E_{\Sigma_t}(x) > cg_\beta(1 + \abs{x}^2 + 2n(t-1))$ for all $t \in [1,T)$ and $x \in \Sigma_t$.
		\item If $T < \infty$, we have 
		\begin{align*}
			\lim_{t \to T} \sup_{\Sigma_t \cap B_{N'\sqrt{t}}} \abs{A_{\Sigma_t}} = \infty.
		\end{align*}
	\end{enumerate}
\end{thm}
\begin{proof}
	Consider the map $\Phi: \Sigma \times (-\varepsilon,\varepsilon) \to \mathbb{R}^{n+1}$ given by
	\begin{align*}
		\Phi(x,\lambda) = x + \lambda \nu_\Sigma.
	\end{align*}
	Since $\Sigma$ is asymptotically conical, we can choose $\varepsilon$ sufficiently small so that the above map is a diffeomorphism onto its image for every $\lambda \in (-\varepsilon,\varepsilon)$. Using this parametrization we can invoke standard existence theorem for MCF to conclude that there exists a unique MCF starting from $\Sigma_1 = \Sigma$. For the properties, Item (1) and (3) follow from \cref{interior-estimates}, and Item (2) is Lemma 5.4 of \cite{BWTopologicalUniqueness}. 
\end{proof}

Given an unstable self-expander $\Sigma$ asymptotic to $\mathcal{C}$, we wish to apply \cref{existence-theorem} to the perturbed self-expander $\Sigma^\varepsilon$ defined by
\begin{align*}
	\Sigma^\varepsilon = \Psi^\varepsilon(\Sigma) \text{ where } \Psi^\varepsilon(x) = x + \varepsilon f(x) \nu_\Sigma.
\end{align*}
It remains to check that $\Sigma^\varepsilon$ satisfies the assumption of \cref{existence-theorem}. We need the following $C^0$-estimate on the first eigenfunction $f$.
\begin{lem}[Proposition 3.2 of \cite{BWMountainPass}]
	\label{eigenfunction-decay-estimates}
	Let $\Sigma$ be an unstable connected self-expander. Let $f$ be the unique positive eigenfunction of $L_\Sigma$ with eigenvalue $\mu_1 < 0$ and $\norm{f}_{W^0_{1/4}(\Sigma)} = 1$. Then $f$ satisfies the following $C^0$ estimate:
	\begin{align*}
		C^{-1}(1 + \abs{x}^2)^{-\frac{1}{2}(n+1 - 2\mu_1)} e^{-\frac{\abs{x}^2}{4}} \le f \le C(1 + \abs{x}^2)^{-\frac{1}{2}(n+1 - 2\mu_1)} e^{-\frac{\abs{x}^2}{4}}.
	\end{align*}
	where $C = C(\Sigma)$.
\end{lem}
\cref{eigenfunction-decay-estimates} and \cref{expander-curvature-estimates} imply that there exists $N > 0$ such that $\Sigma^\varepsilon \setminus B_{NR}(0) \subset \mathcal{T}_{R^{-1}}(\mathcal{C})$ for $R > 1$, so the first condition to apply \cref{existence-theorem} is satisfied. Finally we need to show that perturbing by the first eigenfunction produces a hypersurface with positive expander mean curvature. This relies on the fact that $-L_\Sigma$ is the linearization of the self-expander equation.
\begin{lem}
	\label{expander-mean-convexity}
	Let $\Sigma$ be a connected self-expander $C^{2,\alpha}$-asymptotic to $\mathcal{C}$. Let $f$ be the unique positive eigenfunction corresponding to $\mu_1$ of $L_\Sigma$ with $\norm{f}_{W^0_{1/4}(\Sigma)} = 1$. Then there exists $\varepsilon_0 > 0$ such that for all $\abs{\varepsilon} < \varepsilon_0$ there is $\beta = \beta(\varepsilon)$ such that
	\begin{align*}
		E_{\Sigma^\varepsilon}(x) \ge cg_\beta(1+\abs{x}^2).
	\end{align*}
	Here $\Sigma^\varepsilon$ is the image of $\Sigma$ under the map $\Phi(x) = x + \varepsilon f(x)\nu_\Sigma$.
\end{lem}
\begin{proof}
	By Lemma A.2 of \cite{BWRelativeEntropy} we have (the computation is long, but the result should be standard),
	\begin{align*}
		E_{\Sigma^\varepsilon}(x) = -\varepsilon L_\Sigma f + \varepsilon^2 Q(f,\inner{x,\nabla_\Sigma f},\nabla_\Sigma f,\nabla^2_\Sigma f)
	\end{align*}
	for some homogeneous quadratic polynomial $Q$ with bounded coefficients. When $\varepsilon > 0$, it follows from \cref{eigenfunction-decay-estimates}, that, up to further shrinking $\varepsilon$,
	\begin{align*}
		E_{\Sigma^\varepsilon}(x) \ge \varepsilon\mu_1C^{-1}(1+\abs{x}^2)^{-\frac{1}{2}(n+1 - 2\mu_1)} e^{-\frac{1+\abs{x}^2}{4}} \ge cg_\beta(1 + \abs{x}^2)
	\end{align*}
	where $\beta = \frac{1}{2}(n+1 - 2\mu_1) > 0$. The case $\varepsilon < 0$ can be handled similarly. 
\end{proof}

\cref{expander-mean-convexity} shows that $\Sigma^\varepsilon$ satisfies the second condition of \cref{existence-theorem} for sufficiently small $\varepsilon$. As such, \cref{existence-theorem} can be applied to conclude the short-time existence of an expander mean-convex MCF starting from $\Sigma^\varepsilon$.

\section{Maximum principles}
\label{maximum-principles-appendix}
Here we record the maximum principles from Section 3 of \cite{CHHW} for Brakke flows. These are the essential tools in applying the moving plane method without smoothness. If $\mathcal{M}$ is an integral Brakke flow and $X = (x_0,t_0) \in \supp \mathcal{M}$, the Gaussian density at $X$ is 
\begin{align*}
	\Theta_\mathcal{M}(X) = \lim_{\rho \to 0} \frac{1}{(4\pi \rho^2)^{n/2}} \int_{\mathcal{M}_{t_0 - \rho^2}} e^{-\frac{\abs{x-x_0}^2}{4\rho^2}} d\mathcal{H}^n.
\end{align*}
$\Theta_\mathcal{M}(X)$ is well-defined by Huisken's monotonicity formula. Observe that an entropy upper bound automatically gives upper bounds on all Gaussian densities. 
\begin{thm}[Maximum principle for Brakke flows, Theorem 3.4 of \cite{CHHW}]
	\label{maximum-principles}
	Let $\mathcal{M}$ be a smooth MCF defined in a parabolic ball $P(X,r)$, where $X = (x_0,t_0) \in \supp \mathcal{M}$ and $r > 0$ is sufficiently small such that $\supp \mathcal{M}$ separates $P(x,r)$ into two open connected components $U$ and $U'$. Let $\mathcal{M}'$ be an integral Brakke flow in $P(X,r)$ with $X \in \supp \mathcal{M}'$ and Gaussian density $\Theta_X(\mathcal{M}') < 2$. If $\supp \mathcal{M}' \subset U \cup \supp \mathcal{M}$, then $X$ is a smooth point for $\mathcal{M}'$, and $\mathcal{M}'$ agrees with $\mathcal{M}$ in a small parabolic ball. 
\end{thm}
\begin{thm}[Hopf lemma for tame Brakke flows, Theorem 3.19 of \cite{CHHW}]
	\label{hopf-lemma}
	Let $\mathcal{M}$ and $\mathcal{M}'$ be two integral Brakke flows defined in a parabolic ball $P(X,r)$ where $X = (x_0,t_0) \in \supp \mathcal{M} \cap \mathcal{M}'$. Suppose $X$ is a tame point (see \cref{tameness-defn}) for both $\mathcal{M}$ and $\mathcal{M}'$ and let $\mathbb{H} \subset \mathbb{R}^{n+1}$ be an open half space with $x_0 \in \partial \mathbb{H}$. If in addition $\partial \mathbb{H}$ is not the tangent flow to either $\mathcal{M}$ or $\mathcal{M}'$, and $\reg \mathcal{M}_t \cap \mathbb{H}$ and $\reg \mathcal{M}'_t \cap \mathbb{H}$ are disjoint for $t \in (t_0 - r^2, t_0)$, then $\mathcal{M}$ and $\mathcal{M}'$ are smooth at $(0,0)$ with distinct tangents. 
\end{thm}
We remark that in contrast to the usual smooth maximum principle and Hopf lemma, the smoothness is in fact a conclusion in both of the statements above. 

{\footnotesize
\bibliographystyle{alpha}
\bibliography{symmetry2ref.bib}}
\end{document}